\documentclass[colorinlistoftodos,11pt,a4paper,reqno]{article}
\usepackage{graphicx}
\usepackage[margin=3cm,includehead]{geometry}

\usepackage{amssymb}
\usepackage{amsmath}
\usepackage{stmaryrd}
\usepackage{latexsym}
\usepackage{amscd}
\usepackage{amsthm}
\usepackage{mathrsfs}
\usepackage{url}
\usepackage{amsmath,mathtools}
\usepackage[T1]{fontenc}
\usepackage{tikz-cd}
\usepackage{titlesec}
\usepackage[titletoc,toc,title]{appendix}
\graphicspath{}
\usepackage{fancyhdr} 
\usepackage{esint}

\newcommand\shorttitle{Gromov--Hausdorff instability of the positive mass theorem} 
\newcommand\authors{Romy Marie Merkel} 

\titleformat{\section}  
{\normalfont\large\bfseries\center}{\thesection.}{1em}{}

\makeatletter
\newcommand*{\rom}[1]{\expandafter\@slowromancap\romannumeral #1@}
\makeatother

\usepackage{thmtools}
\newtheorem{satz}{Theorem}
\newtheorem*{satz*}{Theorem}

\newtheorem{lemma}[satz]{Lemma}

\theoremstyle{definition}

\newtheorem*{ack}{Acknowledgments}
\newcommand{\defwort}[1]{\textbf{#1}}

\usepackage[english]{babel}
\addto\extrasenglish{

}

\usepackage{enumitem}

\usepackage[
style=alphabetic,
natbib=true, 
hyperref=true 
]{biblatex} 

\bibliography{tentacle.bib}

\usepackage{amsmath}
\usepackage{upgreek}
\usepackage{commath}
\usepackage{amssymb}
\usepackage{amsthm}
\usepackage{csquotes}
\usepackage{mathtools}
\allowdisplaybreaks

\usepackage{graphicx,tikz,tabularx} 
\usepackage[draft=false,babel,tracking=true,kerning=true,spacing=true]{microtype} 

\usetikzlibrary{calc}

\usetikzlibrary{cd}

\DeclareMathAlphabet{\mathbbmsl}{U}{bbm}{m}{sl}

\newcommand\vol{\text{\normalfont vol}}

\newcommand\re{\text{\normalfont Re}}

\newcommand\R{\mathbb{R}}
\newcommand\C{\mathbb{C}}

\newcommand\N{\mathbb{N}}

\newcommand\Tr{\text{Tr}\,}

\newcommand{\image}{\mathrm{im}}
\newcommand{\eucl}{\mathrm{Eucl}}
\newcommand{\m}{\ensuremath{\mathfrak{m}}}
\DeclareMathOperator{\area}{Area}
\newcommand{\partialbar}{\bar{\partial}}
\newcommand{\D}{\, \mathrm{d}}
\newcommand{\mass}{\hat{\m}}

\newcommand{\diam}{\mathrm{diam}}
\newcommand{\dis}{\mathrm{dis} \, }

\newcommand\Thm{Thm.~}

\newcommand\Def{Def.~}

\newcommand\eg{e.g., }
\newcommand\cf{cf.~}

\DeclareMathSymbol{\shortminus}{\mathbin}{AMSa}{"39}

\usepackage{hhline}

\usepackage{todonotes}
\setuptodonotes{color=red!10}

\usepackage[svgnames, x11names]{xcolor}
\usepackage[hypertexnames=false]{hyperref}
\hypersetup{
	linktocpage=true,
	colorlinks = true,
	linkcolor = Red2,
	anchorcolor = black,
	citecolor = [RGB]{24, 40, 219},
	filecolor = black,
	urlcolor = black
}

\colorlet{RED2}{Red2}

\begin{document}

	\title{Gromov--Hausdorff instability of the positive mass theorem for Kähler manifolds}
	\author{Romy Marie Merkel}
	\date{October 26, 2026}
	\maketitle

\begin{abstract}
	\noindent We construct a sequence of AE Kähler metrics on $\C^2$ with nonnegative and integrable scalar curvature whose ADM mass tends to zero but which, for a suitable choice of basepoints, fails to converge to Euclidean space in the pointed Gromov--Hausdorff sense. 
\end{abstract}

	\section{Introduction}
	
	A complete connected non-compact Riemannian manifold $(M^n,g)$ of dimension $n \geq 3$
	is called \defwort{asymptotically Euclidean} (AE) if there exists a compact subset $K \subset M$ whose complement $M \setminus K$
	 consists of finitely many components $\{M_l\}_{l=1}^L$, each of which is diffeomorphic to 
	$\R^n \setminus \overline{{B}_1(0)}$ and can be equipped with coordinates $(x_1, \dots, x_n)$ such that 
	\begin{align*}
		\big\lvert \partial^I (g_{ij} - \delta_{ij}) \big\rvert 
		= \mathcal{O}\bigl(\lvert x \rvert^{1 - \frac{n}{2}  - \lvert I \rvert - \delta}\bigr)
	\end{align*}
	for all multi-indices $I$ of order $\lvert I \rvert \in \{0, 1, 2\}$ and some $\delta > 0$. We call $\{M_l\}_{l=1}^L$ the \defwort{ends} of $M$. 
	
	At a given end, the \defwort{ADM mass} of $(M^n,g)$ is defined as 
		\begin{align*}
		\m(g) \coloneqq \lim_{\rho \rightarrow \infty} \frac{\Gamma(\frac{n}{2})}{4(n-1)\pi^{n/2}} \int_{S_\rho^{n-1}} (g_{kl, k} - g_{kk,l}) \nu^l \D A_\eucl,
	\end{align*}
	where
	$\nu$ is the outward-pointing Euclidean unit normal to the Euclidean sphere $S_\rho^{n-1} \subset \R^n$ of radius $\rho$.
	\citeauthor{Bartnik86} \autocite{Bartnik86} and \citeauthor{Chrusciel86} \autocite{Chrusciel86} independently proved that this term is finite and independent of the choice of asymptotic coordinates, provided that the scalar curvature $R$ is integrable.

The positive mass theorem states that if $R$ is integrable and nonnegative, then $\m(g)$ is nonnegative for every end, and equals zero for some end if and only if $(M^n,g)$ is isometric to Euclidean space. 
It was first proved by Schoen and Yau \autocite{SY79} for $n= 3$ and by Witten \autocite{Witten81} for $M$ spin. Since then, there have been a number of proofs using different techniques
(see, e.g., \autocite{Brendle26} and the references therein).

The rigidity statement naturally raises the question of stability: Does small mass imply that $(M^n,g)$ is close to Euclidean space in some sense? 
\citeauthor{LS14} \autocite{LS14} showed that this fails in the pointed Gromov--Hausdorff sense, even under rotational symmetry. For all $n \geq 3$, they constructed a sequence of rotationally symmetric AE $n$-manifolds $((M_k^n, g_k))_k$ with nonnegative scalar curvature whose masses tend to zero, but which develop arbitrarily deep gravitational wells that prevent convergence to Euclidean space.
However, the Bray--Kazaras--Khuri--Stern integral inequality \autocite{BKKS22} opened up a new line of investigation.
They showed that the mass of an AE $3$-manifold can be bounded from below by an integral involving the scalar curvature and the Hessian of certain harmonic functions. 
Building on this, Dong and Song \autocite{DS25} proved a stability result that accounts for gravitational wells, but which is available only in dimension $3$ because of the argument used in \autocite{BKKS22}:
If the mass of a sequence of AE $3$-manifolds with nonnegative scalar curvature tends to zero, then there exists a sequence of respective subsets such that their boundary areas tend to zero and their complements converge to Euclidean $3$-space in the pointed Gromov--Hausdorff sense. 
See also \autocite{Sormani26} for a survey of stability results in dimension $3$.

In this paper, we focus on the Kähler case, where
Hein and LeBrun \autocite{HL16} gave a proof of the positive mass theorem based on an explicit formula for the mass. Along the way, they showed that an AE Kähler manifold has only one end, which justifies speaking of \textit{the} mass without specifying an end. 
Their formula depends only on the underlying smooth manifold, the first Chern class of the complex structure and the Kähler class of the metric.

Inspired by \autocite{BKKS22}, Klemmensen \autocite{JohanPMT} showed that the mass of an AE Kähler manifold can be bounded from below by an integral of the scalar curvature and the Hessian of certain holomorphic coordinate functions coming from the complex coordinates at infinity. In contrast to the Riemannian case, this integral inequality holds in all complex dimensions. 
He used this to prove two stability results for the positive mass theorem for Kähler manifolds. While one of them requires a lower bound on the Ricci curvature, the other one applies to more general sequences but requires the excision of sets whose boundaries vanish in the limit:

	\begin{satz}[{\autocite[\Thm B]{JohanPMT}}]\label{thm:Johan}
		Let $((X_k^{2m},g_k, J_k))_{k \in \N}$ be a sequence of AE Kähler manifolds of complex dimension $m$ with nonnegative and integrable scalar curvature, and suppose that the ADM masses satisfy $\m(g_k) \to 0$ as $k\rightarrow \infty$. Then there exist domains $Z_k\subset X_k$, $k \in \N$, that satisfy the following two properties:
		First, for any continuous function $\xi: (0,\infty) \to (0,\infty)$ with $\lim_{x\rightarrow 0^+} \xi(x) = 0$, there exists $k_\xi \in \N$ such that 
		\begin{align*}
			\area_{g_k}(\partial Z_k) \leq \frac{\m(g_k)^{\frac{m}{2m-2}+\frac{1}{2}}}{\xi(\m(g_k))}
			\qquad \text{ for all } k \geq k_\xi.
		\end{align*}
		Second, the sequence
		$(X_k\setminus Z_k)_k$ converges to $(\R^{2m},d_{\eucl})$ in the pointed Gromov--Hausdorff sense,
		\begin{align*}
			(X_k\setminus Z_k,\hat d_{g_k},p_k) \, \overset{\mathrm{pGH}}{\longrightarrow} \, (\R^{2m}, d_{\eucl},0),
		\end{align*}
		where $p_k\in X_k\setminus Z_k$ is any choice of basepoint and $\hat d_{g_k}$ is the induced length metric of $g_k$ on $X_k\setminus Z_k$. 
	\end{satz}

	He subsequently constructed three sequences of AE Kähler manifolds whose masses converge to zero to which this stability result applies. These are given by AE Kähler metrics on $(\C^2, J_0)$ with nonnegative and integrable scalar curvatures, where $J_0$ denotes the natural complex structure. However, all of them converge to $(\R^4, d_{\eucl})$ in the pointed Gromov--Hausdorff sense, even though two of them develop a singularity at $0 \in \C^2$. Hence, no gravitational wells appear in the limits, and so the question remained open whether there existed any \enquote{tentacle examples} which do develop such a well and for which the excision in the above statement is crucial. 
	
	In this work, we construct such an example: 
	
	\begin{satz}\label{thm:tentacleex}
		There exists a sequence $(g_{k})_{k}$ of radially symmetric AE Kähler metrics on $(\C^2, J_0)$ with nonnegative and integrable scalar curvature which satisfies $\m(g_k) \rightarrow 0$
		but which, for a suitable choice of basepoints $p_k \in \C^2$, fails to converge to Euclidean space in the pointed Gromov--Hausdorff sense.
	\end{satz}

	We break the construction of this sequence into several steps. 
	In \autoref{sec:CalabiHeinLeBrun}, we combine the Calabi ansatz with Hein--LeBrun's formula for the ADM mass to reduce the problem to a nonlinear ODE for the Kähler potential. This ODE is autonomous on $\R \setminus [t_0 - \varepsilon, t_0 + \varepsilon]$ for some $t_0 \in \R$ and $\varepsilon > 0$, and depends on the rescaled mass $\mass$. In the subsequent three sections, we derive suitable solutions on $[t_0 + \varepsilon, \infty)$, apply standard ODE methods to ensure that the required conditions remain satisfied on the transition interval $[t_0 - \varepsilon, t_0 + \varepsilon]$ for $\varepsilon > 0$ sufficiently small, and then show that there exists a solution among these which extends globally and satisfies the desired conditions beyond $t_0 - \varepsilon$. 
	In \autoref{sec:ProofThm2}, we set $\mass = \frac{1}{k}$ for $k \in \N$ and 
	pick a suitable $t_0 = t_0(\mass)$ such that the corresponding metrics develop a gravitational well at $0 \in \C^2$ 
	and converge to the half-line $[0, \infty)$ instead of Euclidean space if we take all basepoints to be the origin. 
	Finally, we give a heuristic description of the sets that must be removed for this sequence to converge to Euclidean space.

	\begin{ack}
		I am grateful to my advisor Hans-Joachim Hein for introducing me to this question and for his support. I also thank him and Johan Jacoby Klemmensen for suggesting the general idea behind the construction presented here.
		
		The details of the construction and the main arguments in this paper are due to the author. Claude (Anthropic) was occasionally used to improve the wording of selected sentences, to perform literature searches and to assist with the routine arguments for the following steps: monotonicity of $\phi_D$ in $D$ (\autoref{lemma:phirproperties}), lower semi-continuity of $\varepsilon_1$ (\autoref{lemma:bootstrapping}), upgrade from pointwise to uniform convergence (\autoref{lemma:GHconvergence}) and the inequalities leading to the pointwise convergence of $\phi_k$ (\autoref{sec:excision}). All these arguments were checked and revised by the author, who takes full responsibility for the content of the paper.

		The project was funded by the Deutsche Forschungsgemeinschaft (DFG, German Research Foundation) under Germany's Excellence Strategy EXC 2044--390685587, Mathematics Münster: Dynamics--Geo\-me\-try--Structure and by the CRC 1442 \enquote{Geometry: Deformations and Rigidity} of the DFG. 
	\end{ack}

	\section{The Calabi ansatz and Hein--LeBrun's mass formula}\label{sec:CalabiHeinLeBrun}
	
	Following Calabi's method, we focus on radially symmetric Kähler metrics $\omega = i \partial \partialbar (u(t))$ on $(\C^2, J_0)$ with Kähler potential $u(t)$ for $t = \log r^2$ and $r^2 = \lvert z_1\rvert^2 + \lvert z_2 \rvert^2$. Define $\phi \coloneqq u^\prime$.
Then the metric can be written as
\begin{align}\label{eq:coeffgk}
	g_{\alpha \bar{\beta}} =  \partial_\alpha \partial_{\bar{\beta}}( u(t) ) =  e^{-t}\phi  \delta_{\alpha \beta} + e^{-2t} (\phi^\prime - \phi) \bar{z}_\alpha z_\beta , \qquad  \alpha, \beta \in \{1,2\},
\end{align}
and satisfies 
\begin{align*}
	\Tr (g_{\alpha \bar{\beta}} )_{\alpha, \beta } = e^{-t}(\phi^\prime + \phi)
	\qquad \text{ and } \qquad 
	\det (g_{\alpha \bar{\beta}} )_{\alpha, \beta } = e^{-2t} \phi \phi^\prime.
\end{align*}
Thus, the Kähler condition is equivalent to 
\begin{align*}
	\phi = u^\prime > 0 \qquad \text{ and } \qquad \phi^\prime = u^{\prime \prime} > 0.
\end{align*}
Furthermore, standard computations (compare, e.g., \autocite{FYZ16}) reveal that 
 the scalar curvature takes the form
\begin{align*}
	R = - 2 g^{\gamma \bar{\delta}} \partial_\gamma \partial_{\bar{\delta}} \log  \det (g_{\alpha \bar{\beta}} )_{\alpha, \beta } 
	= \frac{2}{\phi \phi^\prime} \biggl(2 \phi - \phi^\prime - \frac{\phi \phi^{\prime \prime}}{\phi^{\prime}}\biggr)^\prime \bigg\vert_{t = \log r^2}.
\end{align*}
Under the assumption that $\omega$ is AE and that $R$ is integrable, we can plug this into
Hein--LeBrun's formula for the ADM mass \autocite[\Thm C]{HL16} to obtain
\begin{align*}
	\m &= \frac{1}{12 \pi^2} \int_{\C^2} R \D\vol_g \\*
	&= \frac{2}{3 \pi^2} \int_{\C^2}  e^{-2t} \biggl(2 \phi - \phi^\prime - \frac{\phi \phi^{\prime \prime}}{\phi^{\prime}}\biggr)^\prime  \bigg\vert_{t = \log r^2} \D\vol_{\eucl} \\*
	&= \frac{\vol_{\eucl}(S^3)}{3 \pi^2}  \int_{-\infty}^\infty \biggl(2 \phi - \phi^\prime - \frac{\phi \phi^{\prime \prime}}{\phi^{\prime}}\biggr)^\prime (s) \D s.
\end{align*} 
For simplicity, we will work with the rescaled mass 
\begin{align*}
	\mass \coloneqq \frac{3 } {2} \m = \int_{-\infty}^\infty \biggl(2 \phi - \phi^\prime - \frac{\phi \phi^{\prime \prime}}{\phi^{\prime}}\biggr)^\prime (s) \D s
\end{align*}
in the following.

The scalar curvature is nonnegative if and only if the function
\begin{align*}
	A \coloneqq 2 \phi - \phi^\prime - \frac{\phi \phi^{\prime \prime}}{\phi^{\prime}}
\end{align*}
is nondecreasing because
\begin{align*}
	A^\prime \geq 0 \iff R = \frac{2 A^\prime}{\phi \phi^\prime} \geq 0.
\end{align*}
Moreover, if we assume that $\lim_{t \rightarrow - \infty} A(t) = 0$, then
$\mass$ can be directly recovered from $A$ via
\begin{align*}
	A(t) = \int^t_{-\infty}  \biggl(2 \phi - \phi^\prime - \frac{\phi \phi^{\prime \prime}}{\phi^{\prime}}\biggr)^\prime (s) \D s
	\rightarrow
	 \mass
	\qquad \text{ as } t \rightarrow \infty.
\end{align*}

This suggests to prescribe $A$ such that the above conditions are satisfied and then show that a suitable global solution $\phi$ to the corresponding ODE exists. 
For constant $A$, the ODE becomes autonomous and therefore considerably easier to analyze. However, $A$ needs to be smooth, so we cannot simply take the Heaviside function which jumps from $0$ to $\mass$ at some point $t_0 \in \R$. Instead, we consider a 1-parameter family of smooth step functions 
$\{A_\varepsilon: \R \rightarrow \R\}_{\varepsilon > 0}$ satisfying 
\begin{align*}
	A_\varepsilon^\prime \geq 0, 
	\qquad 
	A_\varepsilon \vert_{(-\infty, t_0 - \varepsilon]} = 0, 
	\qquad 
	A_\varepsilon \vert_{[t_0 + \varepsilon, \infty)} = \mass.
\end{align*}

Before we start to solve this ODE, let us work out the actual obstruction to global existence of a solution $\phi$ satisfying the Kähler conditions. 

\begin{lemma}\label{lemma:globalobstruction}
	Let $A \in C^\infty(\R)$ 
	and suppose that $\phi$ is a smooth solution to 
\begin{align*}
	2 \phi - \phi^\prime - \frac{\phi \phi^{\prime \prime}}{\phi^{\prime}} = A
\end{align*}
with $\phi, \phi^\prime \in (0, \infty)$ on the maximal interval $(T, \infty) \subset \R$. Then either $T = - \infty$ or $\lim_{t \rightarrow T^+} \phi(t) = 0$.
\end{lemma}

\begin{proof}
	Integrating gives 
	\begin{align*}
		 &\multicolumn{3}{l}{$\displaystyle (\log \phi^\prime)^\prime = \frac{2 \phi - \phi^\prime - A}{\phi}$ } \\
		&\qquad \iff 
		&\log \frac{\phi^\prime(t)}{\phi^\prime(T+1)} 
		&= \int^t_{T+1} \frac{2 \phi - \phi^\prime - A}{\phi}(s) \D s \\
		&&&= 2\bigl(t- (T+1)\bigr) - \log \frac{\phi(t)}{\phi(T+1)} - \int^t_{T+1} \frac{A(s)}{\phi(s)} \D s 
		\\
		& \qquad \iff  &
		\phi^\prime(t) &= \phi^\prime(T+1) \frac{\phi(T+1)} {\phi(t)} e^{- 2(T+1 - t)} \exp \biggl(\int^{T+1}_t \frac{A(s)}{\phi(s)} \D s \biggr)
	\end{align*}
	for $t \in (T, \infty)$.
	Suppose $T > - \infty$ and $\lim_{t \rightarrow T^+} \phi(t) >0$.
	Then the right-hand side remains bounded within $(0, \infty)$ as $t\rightarrow T^+$, which implies that $ \lim_{t \rightarrow T^+} \phi^\prime(t) \in (0, \infty)$. On the other hand, $\phi^\prime >0$ prevents $\phi$ from blowing up as $t\rightarrow T^+$. Consequently, $\phi$ smoothly extends to a neighborhood of $T$ while the conditions remain satisfied. This contradicts maximality of $(T, \infty)$.
\end{proof}

\section{Solutions for \texorpdfstring{$A = \mass$}{}}

Fix $\mass > 0$ and $t_0 \in \R$. 
To begin with, we solve the ODE on $[t_0, \infty)$ for $A = \mass$.
Integrating gives
\begin{align*}
	2 \phi - \phi^\prime - \frac{\phi \phi^{\prime \prime}}{\phi^{\prime}} = \mass
	\iff \bigl(\phi^2 - \phi \phi^\prime - \mass \phi \bigr)^\prime = 0
	\iff \phi^\prime = \frac{F_D(\phi)}{\phi},
\end{align*}
where 
$F_D(\phi) = (\phi - \frac{\mass}{2})^2 + D$ for some constant $D \in \R$ (\cf \autocite{FYZ16}). 
This defines a Kähler metric if and only if
\begin{align*}
	\phi > 0 \qquad \text{ and } \qquad F_D(\phi) > 0.
\end{align*}
These conditions are automatically satisfied for $\phi(t_0) > 0$ and $D > 0$, in which case
\begin{align*}
	t + C 
	= \int \frac{\phi}{F_D(\phi)} \D \phi
	= \frac{1}{2} \log F_D(\phi) + \frac{\mass}{2 \sqrt{D}} \arctan\biggl(\frac{\phi - \frac{\mass}{2} }{\sqrt{D}}\biggr)
\end{align*}
for some constant $C \in \R$.

\begin{lemma}\label{lemma:phirproperties}
	The unique maximal function $\phi_{D} >0$ defined via
	\begin{align}\label{eq:implicitphir}
		t + \log \frac{1}{2} = \frac{1}{2} \log F_D(\phi_D) + \frac{\mass}{2 \sqrt{D}} \biggl(\arctan\biggl(\frac{\phi_D - \frac{\mass}{2} }{\sqrt{D}}\biggr) - \frac{\pi}{2}\biggr)
	\end{align}
	depends smoothly on $(D,t)$ on $\{D>0\}$ and is decreasing in $D$. Moreover, it satisfies the AE condition and is normalized so that $\phi_D = \frac{1}{2}  e^t + \mass +\mathcal{O}(e^{-t})$ as $t \rightarrow \infty$.
\end{lemma}

\begin{proof}
	The function
	\begin{align*}
		G(\phi, D, t) \coloneqq \frac{1}{2} \log F_D(\phi) + \frac{\mass}{2 \sqrt{D}} \biggl(\arctan\biggl(\frac{\phi - \frac{\mass}{2} }{\sqrt{D}}\biggr) - \frac{\pi}{2}\biggr) -  t - \log \frac{1}{2}
	\end{align*}
	is smooth on $\{D > 0\}$ and satisfies 
	\begin{align*}
		\frac{\partial G}{\partial \phi} = \frac{1}{2} \frac{2 (\phi - \frac{\mass}{2})}{F_D(\phi)} + \frac{\mass}{2 {D}} \frac{1}{1 + \Bigl(\frac{\phi - \frac{\mass}{2} }{\sqrt{D}}\Bigr)^2} 
		= \frac{\phi - \frac{\mass}{2}}{F_D(\phi)} + \frac{ \frac{\mass}{2}}{F_D(\phi)} 
		= \frac{\phi }{F_D(\phi)}.
	\end{align*}
	Therefore, given any $(\phi^*, D^*,t^*)$ with $\phi^* \neq 0$ satisfying \eqref{eq:implicitphir}, the implicit function theorem yields a unique function $\phi_D(t)$ with $\phi_{D^*}( t^*) = \phi^*$ which is smooth in $(D, t)$ and exists as long as $\phi_D(t) \neq 0$.
	We concentrate on the positive branch, that is, the case where $\phi^*, \phi_D$ and $\phi_D^\prime$ are positive, and $\image (\phi_D) = (0, \infty)$. 
	
	Since 
	\begin{align*}
		0 
		= \frac{d }{d D} G(\phi_D(t), D, t) = \frac{\partial G}{\partial \phi_D} \frac{\partial \phi_D}{\partial D} +  \frac{\partial G}{\partial D}
		\iff  \frac{\partial \phi_D}{\partial D} = - \biggl( \frac{\partial G}{\partial \phi_D}\biggr)^{-1} \frac{\partial G}{\partial D},
	\end{align*}
	$\frac{\partial \phi_D}{\partial D}$ has the opposite sign of $\frac{\partial G}{\partial D}$.
	Computing the latter and setting $s \coloneqq (\phi_D - \frac{\mass}{2}) / \sqrt{D}$ gives 
	\begin{align*}
		\frac{\partial G}{\partial D}
		&=\frac{1}{2 F_D(\phi_D)} 
		- \frac{\mass}{4 {D}^{3/2}} \biggl(\arctan\biggl(\frac{\phi_D - \frac{\mass}{2} }{\sqrt{D}}\biggr) - \frac{\pi}{2}\biggr) 
		- \frac{\mass}{4 {D}^2} \biggl(\phi_D - \frac{\mass}{2}\biggr)
		\frac{1}{1 + \Bigl(\frac{\phi_D - \frac{\mass}{2} }{\sqrt{D}}\Bigr)^2} \\
		&=\frac{1}{2D(1 + s^2)} 
		- \frac{\mass}{4 {D}^{3/2}} \biggl(\arctan s - \frac{\pi}{2}\biggr) 
		- \frac{\mass}{4 {D}^2} 
		\frac{\sqrt{D} s}{1 + s^2} \\
		&=\frac{\mass}{4 {D}^{3/2}} 
		\biggl(\frac{\frac{2 \sqrt{D}}{\mass} - s}{1 + s^2} 
		-\arctan s + \frac{\pi}{2}
		\biggr).
	\end{align*}
	This function has nonnegative limits as $s \rightarrow \pm \infty$, and is increasing left of $s_0 = - \frac{\mass}{2 \sqrt{D}}$ and decreasing right of it because
	\begin{align*}
		\frac{d}{ds} \biggl(\frac{\frac{2 \sqrt{D}}{\mass} -  s}{ s^2 +1} 
		- \arctan s + \frac{\pi}{2}
		\biggr)
		= \frac{- (s^2 + 1) - \bigl(\frac{2 \sqrt{D}}{\mass} -  s\bigr)2s}{( s^2 +1)^2} - \frac{1}{1 + s^2} 
		= -2 \frac{\frac{2 \sqrt{D}}{\mass}s + 1}{( s^2 +1)^2}.
	\end{align*}
	This forces the function to be positive everywhere, which implies $\frac{\partial G}{\partial D} > 0$ and $\frac{\partial \phi_D}{\partial D} < 0$. Consequently, $\phi_D$ is decreasing in $D$.

	In order to show that the corresponding metric satisfies the AE condition, it suffices to find $\delta > 0$ such that
	\begin{align*}
		\bigg\lvert \partial^I_z \partial^J_{\bar{z}} \biggl((g_D)_{\alpha \bar{\beta}} - \frac{1}{2}\delta_{\alpha \beta}\biggr) \bigg\rvert 
		= \mathcal{O}\bigl(e^{-t(1 + \lvert I \rvert + \lvert J \rvert  + \delta)/2}\bigr)
	\end{align*}
	for all $\alpha, \beta \in \{1, 2\}$ and all multi-indices $I,J$ with $\lvert I \rvert + \lvert J \rvert \in \{0, 1, 2\}$
	as $t \rightarrow \infty$. 
	
	As the second term in \eqref{eq:implicitphir} is bounded, $t \rightarrow \infty$ is equivalent to $\phi_D \rightarrow \infty$. In that case, 
	\begin{align*}
		\frac{1}{2} \log F_D(\phi_D) 
		&= \log \phi_D + \frac{1}{2} \log \biggl(\biggl(1 - \frac{\mass}{2 \phi_D}\biggr)^2 + \frac{D}{\phi_D^2}\biggr)
		=  \log \phi_D + \frac{1}{2} \biggl(- \frac{\mass}{\phi_D} + \mathcal{O}(\phi_D^{-2}) \biggr) \\*
		&=  \log \phi_D - \frac{\mass}{2\phi_D} + \mathcal{O}(\phi_D^{-2})		
	\end{align*}
	and 
	\begin{align*}
		\arctan\biggl(\frac{\phi_D - \frac{\mass}{2} }{\sqrt{D}}\biggr)
		&= \frac{\pi}{2} - \biggl(\frac{\phi_D - \frac{\mass}{2} }{\sqrt{D}}\biggr)^{-1} + \mathcal{O} \biggl(\biggl(\frac{\phi_D - \frac{\mass}{2} }{\sqrt{D}}\biggr)^{-3}\biggr) \\
		&= \frac{\pi}{2} - \frac{\sqrt{D}}{\phi_D}\biggl(1 - \frac{\mass}{2 \phi_D}\biggr)^{-1} + \mathcal{O} \biggl(\frac{1}{\phi_D^3}\biggl(1 - \frac{\mass}{2 \phi_D}\biggr)^{-3}\biggr) \\
		&= \frac{\pi}{2} - \frac{\sqrt{D}}{\phi_D}\bigl(1 +  \mathcal{O}(\phi_D^{-1})\bigr) + \mathcal{O} (\phi_D^{-3}) \\*
		&= \frac{\pi}{2} - \frac{\sqrt{D}}{\phi_D}+ \mathcal{O} (\phi_D^{-2})
	\end{align*}
	yield
	\begin{align*}
		t + \log \frac{1}{2} = \log \phi_D - \frac{\mass}{\phi_D} + \mathcal{O}(\phi_D^{-2}). 
	\end{align*}
	This implies 
	\begin{align*}
		\frac{1}{2} e^t &= \phi_D e^{- \mass / \phi_D} e^{\mathcal{O}(\phi_D^{-2})} 
		= \phi_D \biggl(1 - \frac{\mass}{\phi_D} + \mathcal{O}(\phi_D^{-2})\biggr)\bigl(1 + \mathcal{O}(\phi_D^{-2})\bigr)
		= \phi_D - \mass + \mathcal{O}(\phi_D^{-1}),
	\end{align*}
	hence
	\begin{align*}
		\phi_D &= \frac{1}{2}  e^t + \mass + \mathcal{O}(\phi_D^{-1}) = \frac{1}{2}  e^t + \mass + \mathcal{O}(e^{-t}), \\
		\phi_D^\prime   &
		= \phi_D - \mass + \frac{\frac{\mass^2}{4} + D}{\phi_D}
		=  \frac{1}{2} e^t + \mathcal{O}(e^{-t}), \\
		\phi_D^{\prime \prime}   &= \phi_D^\prime +\phi_D^\prime  \frac{d}{d \phi_D}\frac{\frac{\mass^2}{4} + D}{\phi_D} = \phi_D^\prime \biggl(1 - \frac{\frac{\mass^2}{4} + D}{\phi_D^2} \biggr) \\
		&= \biggl( \frac{1}{2} e^t + \mathcal{O}(e^{-t})\biggr) \bigl(1 + \mathcal{O}(e^{-2t})\bigr) 
		=  \frac{1}{2} e^t + \mathcal{O}(e^{-t}), \\
		\phi_D^{\prime \prime \prime}   &= \phi_D^{\prime \prime} \biggl(1 - \frac{\frac{\mass^2}{4} + D}{\phi_D^2} \biggr) + (\phi_D^\prime)^2 \mathcal{O}(e^{-3t})
		=  \frac{1}{2} e^t + \mathcal{O}(e^{-t}).
	\end{align*}
	Plugging this into the formula for $(g_D)_{\alpha \bar{\beta}}$ \eqref{eq:coeffgk} gives 
	\begin{align*}
		(g_D)_{\alpha \bar{\beta}} = e^{-t}\phi_D \delta_{\alpha \beta}  + e^{-2t} (\phi_D^\prime - \phi_D) \bar{z}_\alpha z_\beta 
=\frac{1}{2} \delta_{\alpha \beta} + \mathcal{O}(e^{-t}),
	\end{align*}
	and, consequently, $(g_D)_{\alpha \bar{\beta}} - \frac{1}{2} \delta_{\alpha \beta} =  \mathcal{O}(e^{-t(1 + \delta)/2})$ for $\delta \leq 1$.
	As for the cases $\lvert I \rvert + \lvert J \rvert \in \{1,2\}$, $\partial_z^I \partial_{\bar{z}}^J t = \mathcal{O}(e^{-t(\lvert I \rvert + \lvert J \rvert)/2})$ and $\partial_z^I \partial_{\bar{z}}^J \bar{z}_\alpha z_\beta = \mathcal{O}(e^{t(2 - \lvert I \rvert - \lvert J \rvert)/2})$ yield $\partial_z^I \partial_{\bar{z}}^J (g_D)_{\alpha \bar{\beta}} = \mathcal{O}(e^{-t(2 +\lvert I \rvert + \lvert J \rvert)/2})$. This lies in $\mathcal{O}(e^{-t(1 + \lvert I \rvert + \lvert J \rvert  + \delta)/2})$ for $\delta \leq 1$, which shows that the AE condition is fulfilled for any $\delta \in (0, 1]$. 
\end{proof}

	For $D >0$, the Kähler condition on $[t_0, \infty)$ is equivalent to $\phi_D(t_0) > 0$. Therefore, we want to gather more information about the behavior of $\phi_D$ at $t_0$. 
	
	\begin{lemma}\label{lemma:c2Kähler}
		For all $\vert \sigma \rvert < \frac{\mass}{4}$, there exists $D > 0$ such that $\phi_D(t_0) = \frac{D}{\mass} + \frac{\mass}{4}+ \sigma > 0$ and $\mathrm{sgn}(\phi_D(t_0) -  \phi_D^\prime(t_0) ) = \mathrm{sgn}(\sigma)$.
		It satisfies $D \sim\frac{\mass^2 \pi^2 }{4t_0^2}$ as $t_0 \rightarrow - \infty$.
	\end{lemma}

	\begin{proof}
		Plugging $\phi_D(t_0) = \frac{D}{\mass} + \frac{\mass}{4} + \sigma$ into the implicit formula \eqref{eq:implicitphir} for $\phi_D$ gives 
		\begin{align}\label{eq:implicitphit0}
			t_0 + \log \frac{1}{2} = \frac{1}{2} \log \biggl( \biggl( \frac{D}{\mass} - \frac{\mass}{4} + \sigma \biggr)^2  + D \biggr) 
			+ \frac{\mass}{2 \sqrt{D}} \biggl(\arctan\biggl(\frac{\frac{D}{\mass} - \frac{\mass}{4} + \sigma  }{\sqrt{D}}\biggr) - \frac{\pi}{2}\biggr).
		\end{align}
		The right-hand side can be viewed as a smooth function of $D \in (0, \infty)$. We will show that this function goes to $- \infty$ and $+ \infty$ as $D$ approaches the left and right boundary of the interval, respectively, so that the intermediate value theorem implies that it must be surjective onto $\R$ for all $\vert \sigma \rvert < \frac{\mass}{4}$. Then, given any such $\sigma$, there exists $D >0$ such that $\phi_D(t_0) =\frac{D}{\mass} + \frac{\mass}{4} + \sigma > 0$. 
		Furthermore,
		\begin{align*}
			\frac{\phi_D^\prime(t_0) }{\phi_D(t_0)}
			= \frac{F\bigl(\phi_D(t_0)\bigr)}{\phi_D(t_0)^2}
			= 1 +\frac{- \mass \phi_D(t_0) + \frac{\mass^2}{4} + D}{\phi_D(t_0)^2} 
			= 1- \sigma \frac{\mass}{\phi_D(t_0)^2}
			\lesseqgtr 1 
			\iff \sigma \gtreqless 0
		\end{align*}
		proves the second claim.
		
		Therefore, it only remains to show that
		\begin{align*}
			t_0 \rightarrow \begin{cases}
				+ \infty & \text{as } D \rightarrow  \infty, \\
				- \infty & \text{as } D \rightarrow 0.
			\end{cases}
		\end{align*}
		As $D \rightarrow \infty$, the first term in \eqref{eq:implicitphit0} goes to $ +\infty$, while the second one tends to zero, thus $t_0 \rightarrow + \infty$. 
		As $D$ approaches the left boundary, 
		\begin{align*}
			\biggl( \frac{D}{\mass} - \frac{\mass}{4} + \sigma \biggr)^2  + D \rightarrow \biggl( - \frac{\mass}{4} + \sigma \biggr)^2 \in \biggl(0, \frac{\mass^2}{4}\biggr)
		\end{align*}
		and 
		\begin{align*}
			\frac{D}{\mass} - \frac{\mass}{4} + \sigma \rightarrow - \frac{\mass}{4} + \sigma \in \biggl(- \frac{\mass}{2}, 0 \biggr).
		\end{align*}
		Consequently, the first term remains bounded while the second term tends to $- \infty$, so $t_0 \rightarrow - \infty$.
		
		In fact, note that $t_0 \rightarrow - \infty $ even requires $D \rightarrow 0$ because the first term is bounded from below as
		\begin{align*}
			\biggl( \frac{D}{\mass} - \frac{\mass}{4} + \sigma \biggr)^2 + D 
			&=  \frac{D^2}{\mass^2} + \frac{2 D}{\mass}\biggl( - \frac{\mass}{4} + \sigma\biggr) + \biggl( - \frac{\mass}{4} + \sigma \biggr)^2 + D \\
			&=  \frac{D^2}{\mass^2} + \frac{2 D}{\mass}\biggl(\frac{\mass}{4} + \sigma\biggr) + \biggl( - \frac{\mass}{4} + \sigma \biggr)^2 \\
			&> \biggl( - \frac{\mass}{4} + \sigma \biggr)^2 > 0,
		\end{align*}
		and so $t_0 \rightarrow - \infty$ is only possible if 
		$D \rightarrow 0$. 
		If that is the case, \eqref{eq:implicitphit0} gives
		\begin{align*}
			t_0 + \log \frac{1}{2} 
			&= \mathcal{O}(1) 	+ \frac{\mass}{2 \sqrt{D}} \biggl(- \pi + \mathcal{O}\biggl(\biggl(\frac{ \frac{D}{\mass} - \frac{\mass}{4} + \sigma }{\sqrt{D}}\biggr)^{-1}\biggr)  \biggr) \\
			&=\mathcal{O}(1) + \frac{\mass}{2\sqrt{D}} \bigl(- \pi + \mathcal{O}(\sqrt{D})\bigr) \\*
			&= \mathcal{O}(1) - \frac{\mass \pi }{2\sqrt{D}},
		\end{align*}
			which yields the asymptotic behavior
		\begin{align*}
			D = \frac{\mass^2 \pi^2}{4(t_0 + \mathcal{O}(1))^2} \sim\frac{\mass^2 \pi^2 }{4t_0^2}
		\end{align*}
		as $t_0 \rightarrow - \infty$.
	\end{proof}

\section{The transition interval}

So far, we have constructed functions
$\phi_D : [t_0, \infty) \rightarrow \R$ which satisfy $\phi_D > 0, \phi_D^\prime > 0$ and 
\begin{align*}
	2 \phi_D - \phi_D^\prime - \frac{\phi_D \phi_D^{\prime \prime}}{\phi_D^\prime} 
	=\mass,
\end{align*}
and are uniquely determined by a suitable choice of $D >0$. Since the Kähler conditions are open, each of these solutions extends at least slightly beyond $t_0$.

Given such a $D$, Picard--Lindelöf yields a local solution $\phi_{D, \varepsilon}$ to 
\begin{align*}
	2 \phi_{D, \varepsilon} - \phi_{D, \varepsilon}^\prime - \frac{\phi_{D, \varepsilon} \phi_{D, \varepsilon}^{\prime \prime}}{\phi_{D, \varepsilon}^\prime} 
	=A_\varepsilon
	\iff 
	(\log \phi_{D, \varepsilon}^\prime)^\prime = \frac{2 \phi_{D, \varepsilon} - \phi_{D, \varepsilon}^\prime - A_\varepsilon}{\phi_{D, \varepsilon}}
\end{align*}
satisfying $\phi_{D, \varepsilon}(t_0 + \varepsilon) = \phi_D(t_0 + \varepsilon)$ and $\phi_{D, \varepsilon}^\prime(t_0 + \varepsilon) = \phi_D^\prime(t_0 + \varepsilon)$ for every $\varepsilon > 0$. 
By uniqueness, $\phi_{D, \varepsilon} = \phi_D$ on $[t_0+ \varepsilon , \infty)$. 
In particular, $\phi_{D, \varepsilon}$ extends to $(T_{D, \varepsilon}, \infty) \rightarrow \R$ for a $T_{D, \varepsilon} \in [- \infty, t_0 + \varepsilon)$. 
For convenience, we set
\begin{align*}
	p_D &\coloneqq \phi_D(t_0) > 0, 
	\quad &q_D &\coloneqq \phi_D^\prime(t_0) > 0, \\*
	P_{D, \varepsilon} &\coloneqq \phi_D(t_0 + \varepsilon) = \phi_{D, \varepsilon}(t_0 + \varepsilon), 
	\quad &Q_{D, \varepsilon} &\coloneqq \phi_D^\prime(t_0 + \varepsilon) = \phi_{D, \varepsilon}^\prime(t_0 + \varepsilon).
\end{align*}

We will show that $\phi_{D, \varepsilon}$ survives the transition interval if it is short enough and that that upper bound on $\varepsilon$ can be chosen to be lower semi-continuous in $D$. To do so, we use a bootstrapping argument based on the fact that the fundamental theorem of calculus turns bounds on $\phi_{D, \varepsilon}^\prime$ into bounds on $\phi_{D, \varepsilon}$, while the ODE yields the reverse implication.

\begin{lemma}\label{lemma:bootstrapping}
	Fix $0 < l_1 < 1 < l_2$ and let $\widetilde{D} > 0$ be such that $\phi_{\tilde{D}}, \phi_{\tilde{D}}^\prime \in (0, \infty)$ on $(t_0 - \tilde{\varepsilon}, \infty)$ for $\tilde{\varepsilon} > 0$. 
	Then there exists a lower semi-continuous function $\varepsilon_M: (0, \tilde{D}] \rightarrow (0, \tilde{\varepsilon}]$ such that 
	for all $D \in (0, \tilde{D}]$
	 and $0 < \varepsilon < \varepsilon_M(D)$,
	 	\begin{align}\label{eq:boundstransition}
	 	l_1 p_D  < \phi_{D,  \varepsilon} < l_2 p_D  \qquad \text{ and } \qquad l_1 q_D  < \phi_{D,  \varepsilon}^\prime   < l_2 q_D
	 \end{align}
	on 
	$[t_0- \varepsilon, t_0 + \varepsilon]$. 
\end{lemma}

\begin{proof} 
	To begin with, recall that $\phi_{D}$ is decreasing in $D$ by \autoref{lemma:phirproperties}. Therefore, $\phi_D \geq \phi_{\tilde{D}} > 0$ for $0 < D \leq \tilde{D}$, hence \autoref{lemma:globalobstruction} ensures that $\phi_{D}, \phi_{D}^\prime \in (0, \infty)$ on $(t_0 - \tilde{\varepsilon}, \infty)$. 
	
	Fix $D \in (0, \tilde{D}]$ and $\varepsilon  > 0$. 
	We define
	\begin{align*}
		S_{D, \varepsilon} \coloneqq \bigl\{t \in [t_0 - \varepsilon, t_0 + \varepsilon] \mid &\, \phi_{D, \varepsilon} \text{ exists on } [t, t_0 + \varepsilon] \text{ with } l_1 p_D \leq \phi_{D, \varepsilon} \leq l_2 p_D\\
		 &\, \text{and } l_1 q_D \leq \phi_{D, \varepsilon}^\prime \leq l_2 q_D\bigr\}.
	\end{align*}
	Since $\phi_{D, \varepsilon}$ is smooth and the conditions are closed, $S_{D, \varepsilon}$ is a closed subinterval of $[t_0 - \varepsilon, t_0 + \varepsilon]$.
	We need to find conditions on $\varepsilon$ under which bootstrapping gives $\inf S_{D, \varepsilon} = t_0 - \varepsilon$. 
	
	To do so, we must first ensure that $S_{D, \varepsilon} \neq \emptyset$.
	Let $k_1 \coloneqq (l_1 + 1)/2$ and $k_2 \coloneqq (l_2 + 1)/2$. 
	As $\phi_D$ is smooth,
	\begin{align*}
		\varepsilon_1(D) \coloneqq \sup \bigl\{\rho \in (0, \tilde{\varepsilon})\mid &\, k_1 p_D < \phi_D(t) < k_2 p_D  \text{ and }  k_1 q_D < \phi_D^\prime(t) < k_2 q_D \\
		&\, \text{for } \lvert t - t_0 \rvert < \rho \bigr\}
	\end{align*}
	is positive.
	By imposing $\varepsilon < \varepsilon_1(D)$, we achieve that 
	\begin{align*}
		l_1 p_D< k_1 p_D< P_{D, \varepsilon} < k_2p _D< l_2p_D \quad \text{ and } \quad 
		l_1 q_D< k_1 q_D< Q_{D, \varepsilon}  < k_2 q_D < l_2 q_D,
	\end{align*}
	hence $S_{D, \varepsilon} \neq \emptyset$.

	Next, we apply the fundamental theorem of calculus and the ODE to improve the given estimates on $S_{D, \varepsilon}$. For $t \in S_{D, \varepsilon}$, we have 
	\begin{align*}
		\phi_{D, \varepsilon}(t) &= \phi_D(t_0 + \varepsilon) - \int^{t_0 + \varepsilon}_t \phi_{D, \varepsilon}^\prime (s) \D s \geq P_{D, \varepsilon} - \int^{t_0 + \varepsilon}_t l_2 q_D \D s 
		> k_1 p_D - 2 \varepsilon l_2 q_D, \\*
		\phi_{D, \varepsilon}(t) &= \phi_D(t_0 + \varepsilon) - \int^{t_0 + \varepsilon}_t \phi_{D, \varepsilon}^\prime (s) \D s \leq P_{D, \varepsilon} < k_2 p_D < l_2 p_D.
	\end{align*} 
	These are strictly better than the given bounds on $\phi_{D, \varepsilon}$ if and only if
	\begin{align*}
		k_1 p_D - 2 \varepsilon l_2 q_D \geq l_1 p_D \iff \varepsilon \leq \varepsilon_2(D) \coloneqq \frac{(k_1 - l_1) p_D}{2 l_2 q_D}.
	\end{align*}
	
	On the other hand,
	we can bound 
	\begin{align*}
		\lvert (\log \phi_{D, \varepsilon}^\prime)^\prime \rvert
		= \frac{\lvert 2 \phi_{D, \varepsilon} - \phi_{D, \varepsilon}^\prime - A_\varepsilon \rvert}{\lvert \phi_{D, \varepsilon} \rvert}
		\leq \frac{2 l_2 p_D + l_2 q_D + \mass}{l_1 p_D} \eqqcolon K_D
	\end{align*}
	by a constant $K_D$ which is independent of $\varepsilon$. 
	This leads to bounds on $\log \phi_{D, \varepsilon}^\prime$,
	\begin{align*}
		\log \phi_{D, \varepsilon}^\prime(t)
		&= \log \phi_{D, \varepsilon}^\prime(t_0 + \varepsilon) 
		- \int^{t_0 + \varepsilon}_t(	\log \phi_{D, \varepsilon}^\prime)^\prime(s) \D s \\*
		&\geq \log Q_{D, \varepsilon}
		- \int^{t_0 + \varepsilon}_{t_0-\varepsilon}\lvert(	\log \phi_{D, \varepsilon}^\prime)^\prime(s)\rvert \D s \\*
		&\geq \log Q_{D, \varepsilon} - 2 \varepsilon K_D , \\
		\log \phi_{D, \varepsilon}^\prime(t)
		&= \log \phi_{D, \varepsilon}^\prime(t_0 + \varepsilon) 
		- \int^{t_0 + \varepsilon}_t(	\log \phi_{D, \varepsilon}^\prime)^\prime(s) \D s \\*
		&\leq \log Q_{D, \varepsilon}
		+ \int^{t_0 + \varepsilon}_{t_0-\varepsilon}\lvert(	\log \phi_{D, \varepsilon}^\prime)^\prime(s)\rvert \D s \\*
		&\leq \log Q_{D, \varepsilon} + 2 \varepsilon K_D,
	\end{align*}
	hence
	\begin{align*}
		k_1 q_D e^{- 2 \varepsilon K_D } < Q_{D, \varepsilon} e^{- 2 \varepsilon K_D } \leq \phi_{D, \varepsilon}^\prime \leq Q_{D, \varepsilon} e^{2 \varepsilon K_D } < k_2 q_D e^{2 \varepsilon K_D }.
	\end{align*}
	These are strictly better than the given estimates if and only if 
	\begin{align*}
		k_1 q_D e^{- 2 \varepsilon K_D } \geq l_1 q_D \iff - 2 \varepsilon K_D \geq \log \frac{l_1}{k_1} \iff \varepsilon  \leq \varepsilon_3(D) \coloneqq \frac{1}{2K_D} \log \frac{k_1}{l_1}
	\end{align*}
	and 
	\begin{align*}
		k_2 q_D e^{2 \varepsilon K_D } \leq l_2 q_D \iff 2 \varepsilon K_D \leq \log \frac{l_2}{k_2} \iff \varepsilon  \leq \varepsilon_4(D) \coloneqq \frac{1}{2K_D} \log \frac{l_2}{k_2}. 
	\end{align*}

	Consequently, if $\varepsilon < \varepsilon_M(D) \coloneqq \min \{\varepsilon_1, \varepsilon_2, \varepsilon_3, \varepsilon_4\}(D)$,
	then the strict inequalities
	\begin{align*}
		l_1 p_D < \phi_{D, \varepsilon} < l_2 p_D \qquad \text{ and } \qquad l_1 q_D < \phi_{D, \varepsilon}^\prime < l_2 q_D
	\end{align*}
	hold on $S_{D, \varepsilon}$.
	However, these improved conditions are open and $\phi_{D, \varepsilon}$ is smooth, so $S_{D, \varepsilon}$ is not only closed in $[t_0 - \varepsilon, t_0 + \varepsilon]$ but also open. Since $S_{D, \varepsilon}$ is nonempty, this implies $S_{D, \varepsilon} = [t_0 - \varepsilon, t_0 + \varepsilon]$. As a result, $\phi_{D, \varepsilon}$ satisfies \eqref{eq:boundstransition}
	on the entire transition interval.

	It remains to show that $\varepsilon_M = \min \{\varepsilon_1, \varepsilon_2, \varepsilon_3, \varepsilon_4\} \in (0, \tilde{\varepsilon}]$ is lower semi-continuous on $(0, \tilde{D}]$.
	According to \autoref{lemma:phirproperties}, $\phi_D$ depends smoothly on $(D, t)$ as long as $\phi_D> 0$. In that case, $p_D$ and $q_D$ are smooth in $D$ too, which implies that so are $K_D>0$ and $\varepsilon_2, \varepsilon_3, \varepsilon_4$. 
	
	Define
	\begin{align*}
		H(D, t) \coloneqq  \min \bigl\{\phi_D(t) - k_1 p_D, k_2 p_D - \phi_D(t), \phi_D^\prime(t) - k_1 q_D, k_2 q_D - \phi_D^\prime(t)\bigr\}
	\end{align*}
	for $(D, t) \in (0, \tilde{D}] \times (t_0 - \tilde{\varepsilon}, \infty)$. Then $\varepsilon_1$ can be equivalently written as 
	\begin{align*}
		\varepsilon_1(D) = \sup \{\rho \in (0, \tilde{\varepsilon}) \mid H(D, t) > 0 \text{ for } \lvert t-t_0\rvert < \rho\}.
	\end{align*}
	As the minimum of smooth functions, $H$ is continuous. 
	
	Fix $D^* \in (0, \tilde{D}]$ and suppose that $\varepsilon_1(D^*) > a > 0$. To prove that $\varepsilon_1$ is lower semi-continuous, we need to show that this inequality remains satisfied close to $D^*$. Let $\varepsilon_1(D^*) > b > a$. Then $H(D^*, \cdot) > 0$ on $[t_0 - b, t_0 + b]$. By continuity, $H(D^*, \cdot)$ attains a minimum $H^{\min}>0$ on that closed interval. The set $S \coloneqq \{D^*\} \times [t_0 - b, t_0 + b] $ is a slice of $(0, \tilde{D}] \times [t_0 - b, t_0 + b]$ and is contained in the set 
	\begin{align*}
		N \coloneqq \bigl\{(D, t) \in (0, \tilde{D}] \times [t_0 - b, t_0 + b]  \ \big\vert \ H(D, t) > \tfrac{H^{\min}}{2}\bigr\},
	\end{align*}
	which is open in $(0, \tilde{D}] \times [t_0 - b, t_0 + b]$ because $H$ is continuous. According to the tube lemma, there exists $\delta > 0$ such that $S \subset (B_\delta(D^*) \cap (0, \tilde{D}])  \times  [t_0 - b, t_0 + b] \subset N$. This implies $H(D, t) > 0$ for $(D, t) \in (B_\delta(D^*) \cap (0, \tilde{D}]) \times [t_0 - b, t_0 + b]$, hence $\varepsilon_1(D) \geq b > a$ for $D \in (0, \tilde{D}]$ with $\lvert D - D^* \rvert < \delta$. 
	
	As the minimum of one lower semi-continuous and three smooth functions, $\varepsilon_M$ is lower semi-continuous in $D$. 
\end{proof}

\section{Solutions for \texorpdfstring{$A = 0$}{}}\label{sec:A0}

Beyond the transition interval, $A = 0$, hence any solution $\phi$ on that ray satisfies 
\begin{align*}
	2 \phi - \phi^\prime - \frac{\phi \phi^{\prime \prime}}{\phi^{\prime}} = 0 
	&\iff (\phi \phi^\prime)^\prime = 2 \phi \phi^\prime 
	\iff 1 = \bigl(\ln (\phi \phi^\prime)^{1/2}\bigr)^\prime \\
	&\iff t + d_1 = \ln (\phi \phi^\prime)^{1/2} 
	\iff e^{2(t+d_1)} = \phi \phi^\prime = \frac{1}{2} (\phi^2)^\prime \\
	&\iff e^{2(t+d_1)} + d_2 = \phi^2
	\iff \phi(t) = \sqrt{e^{2(t+d_1)} + d_2}
\end{align*}
for $d_1, d_2 \in \R$.
The Kähler condition remains fulfilled as $t \rightarrow - \infty$ if and only if $d_2 \geq 0$, but the corresponding metric is only smooth at the origin if $d_2 = 0$. This is because, for $d_2 = 0$, the metric takes the form $g_{\alpha\bar{\beta}} = \delta_{\alpha \beta} e^{d_1} \in C^\infty(\R)$. For $d_2 > 0$, on the other hand, 
$g_{2 \bar{2}} = e^{-t} \phi(t) = \sqrt{e^{2 d_1} + d_2 e^{-2t}}$ on $\{z_2 = 0\}$, which diverges as $t \rightarrow - \infty$.

Suppose $\phi_{D, \varepsilon}$ satisfies the Kähler condition on $[t_0 - \varepsilon, \infty)$. 
By Picard--Lindelöf, $\phi_{D, \varepsilon}(t_0 - \varepsilon) $ and $\phi_{D, \varepsilon}^\prime(t_0 - \varepsilon) $ uniquely determine $d_1, d_2$ and hence how the solution continues beyond $t = t_0 - \varepsilon$. As
\begin{align*}
	\phi^\prime(t) = \frac{e^{2(t+d_1)}}{\sqrt{e^{2(t+d_1)} + d_2}} \gtreqqless \frac{e^{2(t+d_1) } + d_2}{\sqrt{e^{2(t+d_1)} + d_2}} = \phi(t) \iff  d_2 \lesseqqgtr 0
\end{align*}
for any $t$, 
$\phi_{D, \varepsilon}$ gives a globally defined smooth AE Kähler metric if and only if $\phi_{D, \varepsilon}(t_0 - \varepsilon) =\phi_{D, \varepsilon}^\prime(t_0 - \varepsilon)$.

\begin{lemma}\label{lemma:globalKähler}
	There exist $D^* >0$ and $\varepsilon^* > 0$ such that $ \phi_{D^*, \varepsilon^*} , \phi_{D^*, \varepsilon^*}^\prime  \in  (0, \infty)$ on $\R$ and $\phi_{D^*, \varepsilon^*}(t_0 - \varepsilon^*) =\phi_{D^*, \varepsilon^*}^\prime(t_0 - \varepsilon^*)$.
	Moreover, $D^* \sim\frac{\mass^2 \pi^2 }{4t_0^2}$ as $t_0 \rightarrow - \infty$.
\end{lemma}

\begin{proof}
	Set $\sigma^\pm \coloneqq \pm \frac{\mass}{48} $.
	By \autoref{lemma:c2Kähler}, there exist $D^\pm > 0$ such that $p_{D^\pm} = \frac{D^\pm}{\mass} + \frac{\mass}{4}+ \sigma^\pm > 0$.
	In particular, 
	$\phi_{D^\pm}, \phi_{D^\pm}^\prime \in (0, \infty)$ on $(t_0 - \tilde{\varepsilon}, \infty)$ for some $\tilde{\varepsilon} > 0$, 
	and $p_{D^\pm} = (1 \pm \tau^\pm)q_{D^\pm} > 0$ for some $\tau^\pm > 0$ .
	Since 
	\begin{align*}
		\lim_{\delta \rightarrow 0^+}\frac{1-\delta}{1+ \delta} (1+\tau^+) = 1 + \tau^+ > 1
		\quad \text{ and } \quad 
		\lim_{\delta \rightarrow 0^+}\frac{1+\delta}{1- \delta} (1-\tau^-) = 1 - \tau^- < 1 ,
	\end{align*}
	we can pick $0 < \delta < \frac{1}{2}$ small enough such that 
	\begin{align*}
		\frac{1-\delta}{1+ \delta} (1+\tau^+) \geq 1
		\qquad \text{ and } \qquad 
		\frac{1+\delta}{1- \delta} (1-\tau^-) \leq 1.
	\end{align*}
	
	Setting $\frac{1}{2} < l_1 = 1-\delta < 1 < l_2 = 1 + \delta < \frac{3}{2}$ in \autoref{lemma:bootstrapping} implies that 
	$\varepsilon_M$ attains a positive minimum $0 < \varepsilon_M^{\min}\leq \tilde{\varepsilon}$ on the compact interval $C \coloneqq [\min\{D^+, D^-\}, \max\{D^+, D^-\}]$, and that for all $D \in C$ and $0 < \varepsilon < \varepsilon_M^{\min}$,
	\begin{align*}
		(1-\delta) p_D  < \phi_{D, \varepsilon}< (1+\delta) p_D  \qquad \text{ and } \qquad (1-\delta) q_D  < \phi^\prime_{D, \varepsilon}  < (1+\delta) q_D
	\end{align*}
	on $[t_0- \varepsilon, t_0 + \varepsilon]$. 
	Applying this to $D^\pm$ at $t = t_0 - \varepsilon$ gives 
	\begin{align*}
		\phi_{D^+, \varepsilon}(t_0 - \varepsilon) 
		&> (1-\delta) p_{D^+} =   (1-\delta) (1 + \tau^+)q_{D^+} >  \frac{1-\delta}{1+\delta}  (1 + \tau^+) \phi^\prime_{D^+, \varepsilon}(t_0 - \varepsilon) \\*
		&\geq \phi^\prime_{D^+, \varepsilon}(t_0 - \varepsilon)
		\intertext{and}
		\phi_{D^-, \varepsilon}(t_0 - \varepsilon) 
		&< (1+\delta) p_{D^-} =   (1+\delta) (1 - \tau^-)q_{D^-} <  \frac{1+\delta}{1-\delta}  (1 - \tau^-) \phi^\prime_{D^-, \varepsilon}(t_0 - \varepsilon) \\*
		&\leq \phi^\prime_{D^-, \varepsilon}(t_0 - \varepsilon).
	\end{align*}
	Consequently, we have $\phi_{D, \varepsilon} , \phi_{D, \varepsilon}^\prime  \in  (0, \infty)$ on $[t_0 - \varepsilon, \infty)$, as well as 
	\begin{align*}
		\phi_{D^+, \varepsilon}(t_0 - \varepsilon)  > \phi^\prime_{D^+, \varepsilon}(t_0 - \varepsilon)
		\qquad \text{ and } \qquad 
		\phi_{D^-, \varepsilon}(t_0 - \varepsilon) < \phi^\prime_{D^-, \varepsilon}(t_0 - \varepsilon)
	\end{align*}
	for all $D \in C$ and $0 < \varepsilon < \varepsilon_M^{\min}$. 
	\clearpage
	
	As $D$ affects $\phi_{D, \varepsilon}$ only through the initial conditions,
	and solutions of smooth ODEs depend smoothly on $t$ and the latter, $\phi_{D, \varepsilon}(t)$ is smooth in $(D, t)$. Therefore, $\phi_{D, \varepsilon} - \phi_{D, \varepsilon}^\prime$ depends smoothly on $D$, and the intermediate value theorem implies that there exists $D^* \in C$ such that $\phi_{D^*, \varepsilon^*} , \phi_{D^*, \varepsilon^*}^\prime  \in  (0, \infty)$ on $[t_0 - \varepsilon^*, \infty)$ and $\phi_{D^*, \varepsilon^*}(t_0 - \varepsilon^*) = \phi^\prime_{D^*, \varepsilon^*}(t_0 - \varepsilon^*)$ for $\varepsilon^* \coloneqq \varepsilon_M^{\min}/2$. 
	The preceding discussion further implies that $\phi_{D^*, \varepsilon^*} , \phi_{D^*, \varepsilon^*}^\prime  \in  (0, \infty)$ remains true on all of $\R$.

According to \autoref{lemma:c2Kähler}, 
$D^\pm \sim\frac{\mass^2 \pi^2 }{4t_0^2}$ as $t_0 \rightarrow - \infty$.
As $D^*$ lies between $D^+$ and $D^-$, it behaves in the same way. 
\end{proof}

 \section{Proof of \autoref{thm:tentacleex}}\label{sec:ProofThm2}
 
 For all $k \in \N$, \autoref{lemma:globalKähler} provides us with $D_k \coloneqq D^*(\mass_k, t_{0,k}) > 0$ and $\varepsilon_k \coloneqq  \varepsilon^*(\mass_k, t_{0,k}) > 0$ for $\mass_k \coloneqq \frac{1}{k}$ and $t_{0,k} \coloneqq - e^{1/\mass_k} = - e^{k}$ such that $\phi_k \coloneqq \phi_{D^*, \varepsilon^*}$ defines a smooth AE Kähler metric $g_k$ on $(\C^2, J_0)$ with nonnegative and integrable scalar curvature $$R_k = \frac{2 A_{\varepsilon^*}^\prime}{\phi_k\phi_k^\prime}.$$
We will show that this sequence has all the properties required by \autoref{thm:tentacleex}.
Before we do so, let us derive some estimates that we need for the proof.

\begin{lemma}\label{lemma:diamEuclSphere}
	The $g_k$-diameter of the Euclidean sphere
	\begin{align*}
		S^3_\sigma = \{z \in \C^2 \mid r(z) = \sigma\} = \{(z_1, z_2) \in \C^2 \mid \lvert z_1 \rvert^2 + \lvert z_2 \rvert^2= \sigma^2\}
	\end{align*}
	of fixed radius $\sigma  > 0$ is bounded from above by 
	\begin{align*}
		\diam_{g_k} S^3_\sigma \leq \sqrt{\frac{65}{2} \pi^2 \phi_k (\log \sigma^2)
			+ 32 \pi^2 \phi_k^\prime(\log \sigma^2)}.
	\end{align*}
\end{lemma}

\begin{proof}
	 	To begin with, note that every $z \in  S^3_\sigma$ can be written as 
	\begin{align*}
		z = (z_1, z_2) 
		= \bigl(\lvert z_1 \rvert e^{i \alpha}, \lvert z_2 \rvert e^{i \beta}\bigr)
		= (\sigma \cos \theta  e^{i \alpha}, \sigma \sin \theta  e^{i \beta})
	\end{align*}
	for some $\alpha, \beta \in [0, 2\pi)$ and $\theta \in [0, \frac{\pi}{2}]$. 
	Using this, we see that the two points 
	\begin{align*}
		z = (\sigma \cos \theta_0  e^{i \alpha_0}, \sigma \sin \theta_0  e^{i \beta_0})
		\qquad \text{ and } \qquad 
		w = (\sigma \cos \theta_1  e^{i \alpha_1}, \sigma \sin \theta_1  e^{i \beta_1})
	\end{align*}
	in $S^3_\sigma$
	are joined by the smooth path 
	\begin{align*}
		\gamma(s) = \bigl(\sigma \cos \theta(s)  e^{i \alpha(s)}, \sigma \sin \theta(s)  e^{i \beta(s)}\bigr) \in S^3_\sigma,  \qquad  s \in [0,1],
	\end{align*}
	where
	\begin{align*}
		\theta(s) = \theta_0 (1-s) + \theta_1 s, \qquad
		\alpha (s) = \alpha_0 (1-s) + \alpha_1 s, \qquad 
		\beta(s) = \beta_0 (1-s) + \beta_1 s.
	\end{align*}
	By \eqref{eq:coeffgk}, the metric coefficients of $g_k$ along $\gamma$ are given by 
	\begin{align*}
		(g_k)_{1 \bar{1}} &= \sigma^{-2}\bigl(\phi_k + (\phi_k^\prime - \phi_k) \cos^2 \theta \bigr), 
		&(g_k)_{2 \bar{2}} &= \sigma^{-2}\bigl(\phi_k + (\phi_k^\prime - \phi_k) \sin^2 \theta \bigr), \\*
		(g_k)_{1 \bar{2}} &= \sigma^{-2} e^{i(- \alpha + \beta)} (\phi_k^\prime - \phi_k) \cos \theta \sin \theta,&&
	\end{align*}
	where $\phi_k$ and $\phi_k^\prime$ are evaluated at $t = \log \sigma^2$.
	On the other hand,
	\begin{align*}
		\dot{\gamma} = \bigl(\sigma e^{i \alpha} (- \dot{\theta} \sin \theta + i \dot{\alpha} \cos \theta), \sigma e^{i \beta} (\dot{\theta} \cos \theta + i \dot{\beta} \sin \theta)\bigr)
	\end{align*}
	satisfies 
	\begin{align*}
		\lvert \dot{\gamma}_1 \rvert^2 &= \sigma^2 (\dot{\theta}^2 \sin^2 \theta + \dot{\alpha}^2 \cos^2 \theta), 
		\qquad \quad 
		\lvert \dot{\gamma}_2 \rvert^2 = \sigma^2 (\dot{\theta}^2 \cos^2 \theta + \dot{\beta}^2 \sin^2 \theta), \\
		\dot{\gamma}_1 \overline{\dot{\gamma}_2} &= \sigma^{2} e^{i(\alpha -\beta)} \bigl((- \dot{\theta}^2 + \dot{\alpha} \dot{\beta}) \cos \theta\sin \theta + i \dot{\theta}(\dot{\beta} \sin^2 \theta + \dot{\alpha} \cos^2 \theta) \bigr).
	\end{align*}
	Combining this gives 
	\begin{align*}
		\frac{1}{2}g_k(\dot{\gamma}, \dot{\gamma}) 
		&= (g_k)_{1 \bar{1}} \lvert \dot{\gamma}_1 \rvert^2 + (g_k)_{2 \bar{2}}	\lvert \dot{\gamma}_2 \rvert^2
		+ 2 \,\re\bigl((g_k)_{1 \bar{2}} \dot{\gamma}_1 \overline{\dot{\gamma}_2}\bigr)
		\\
		&= \bigl(\phi_k + (\phi_k^\prime - \phi_k) \cos^2 \theta \bigr) (\dot{\theta}^2 \sin^2 \theta + \dot{\alpha}^2 \cos^2 \theta) \\
		&\quad + \bigl(\phi_k + (\phi_k^\prime - \phi_k) \sin^2 \theta \bigr) (\dot{\theta}^2 \cos^2 \theta + \dot{\beta}^2 \sin^2 \theta) \\
		&\quad + 2 (- \dot{\theta}^2 + \dot{\alpha} \dot{\beta}) (\phi_k^\prime - \phi_k) \cos^2 \theta\sin^2 \theta \\
		&= \phi_k (\dot{\theta}^2 + \dot{\alpha}^2 \cos^2 \theta + \dot{\beta}^2 \sin^2 \theta) \\
		&\quad+ (\phi_k^\prime - \phi_k) (\dot{\alpha}^2 \cos^4 \theta + \dot{\beta}^2 \sin^4 \theta + 2 \dot{\alpha}\dot{\beta} \cos^2 \theta \sin^2 \theta) \\
		&= \phi_k \bigl(\dot{\theta}^2 + (\dot{\alpha} - \dot{\beta})^2 \cos^2 \theta \sin^2 \theta\bigr) 
		+ \phi_k^\prime (\dot{\alpha} \cos^2 \theta + \dot{\beta} \sin^2 \theta)^2 \\
		&\leq \phi_k \bigl(\vert\dot{\theta}\vert^2 + (\lvert\dot{\alpha} \rvert + \lvert\dot{\beta}\rvert)^2 \bigr) 
		+ \phi_k^\prime (\lvert\dot{\alpha}\rvert + \lvert\dot{\beta}\rvert)^2 \\
		&\leq \phi_k \bigl(\bigl(\tfrac{\pi}{2}\bigr)^2 + (4 \pi)^2 \bigr) 
		+ \phi_k^\prime (4 \pi )^2\\
		&= \frac{65}{4} \pi^2 \phi_k 
		+ 16 \pi^2 \phi_k^\prime
	\end{align*}
	and, consequently,
	\begin{align*}
		d_{g_k}(z,w) = \int_0^1 \sqrt{g_k(\dot{\gamma}, \dot{\gamma})(s) } \D s \leq \sqrt{\frac{65}{2} \pi^2 \phi_k (\log \sigma^2)
			+ 32 \pi^2 \phi_k^\prime(\log \sigma^2)}.
	\end{align*}
	As the points were arbitrary, this verifies the claimed inequality.
\end{proof}

\begin{lemma}\label{lemma:limitfk}
	Let $k \in \N$ and define $t_k^{(1/2)} \coloneqq \phi_k^{-1}(\frac{\mass_k}{2})$. The
	 function
		\begin{align*}
		f_k: (0, \infty) \rightarrow (0, \infty), \qquad r \mapsto \frac{1}{\sqrt{2}} \int^{\log r^2}_{- \infty} \sqrt{\phi_k^\prime(s)} \D s
	\end{align*}
	satisfies 
	\begin{align*}
		\lim_{k \rightarrow \infty } f_k(e^{(t_{0,k} + \varepsilon_k)/2}) = 0
		\qquad \text{ and } \qquad
		\lim_{k \rightarrow \infty }f_k(e^{t_k^{(1/2)}/2}) = \infty.
	\end{align*}
\end{lemma}

\begin{proof}
	First, we show that
	$\phi_{k}$ satisfies 
	\begin{align}\label{eq:boundsphik}
		\phi_{k}(t_{0,k} + \varepsilon_k) \leq \frac{7}{16} \mass_k < \frac{\mass_k}{2} - \sqrt{D_k} <  \frac{\mass_k}{2}
	\end{align}
	for $k$ sufficiently large. 
To do so, let $\sigma_k^\pm = \pm \frac{\mass_k}{48}$, $D_k^\pm$ and $\delta_k \in (0, \frac{1}{2})$ be as in the proof of \autoref{lemma:globalKähler} applied to 
$(\mass_k, t_{0,k}) = (\frac{1}{k}, -e^k)$, 
and write $D_k^{\min} \coloneqq \min \{D_k^+, D_k^-\}$ and $\sigma_k^{\min} \in \{\sigma_k^+, \sigma_k^-\}$ for the corresponding parameter.
	Then the left-hand side is bounded from above by 
	\begin{align*}
		\phi_{k}(t_{0,k} + \varepsilon_k) < \phi_{D_k^{\min}}(t_{0,k} + \varepsilon_k)
		< (1 + \delta_k)  \biggl(\frac{D_k^{\min}}{\mass_k} + \frac{\mass_k}{4}+ {\sigma_k^{\min}} \biggr)
		< \frac{3}{2} \mass_k \biggl(\frac{D_k^{\min}}{\mass_k^2} + \frac{13}{48}\biggr).
	\end{align*}
Since $D_k, D_k^{\min} \sim\frac{\mass_k^2 \pi^2 }{4t_{0,k}^2} \rightarrow 0$ as $k \rightarrow \infty$ by \autoref{lemma:c2Kähler} and \autoref{lemma:globalKähler}, we have 
$\frac{D_k^{\min}}{\mass_k^2} < \frac{1}{48}$ and $\frac{1}{2} - \frac{\sqrt{D_k}}{\mass_k} > \frac{7}{16}$ for $k$ sufficiently large.
Combining this gives
	\begin{align*}
		\frac{\phi_{k}(t_{0,k} + \varepsilon_k) }{\mass_k} 
		<\frac{3}{2}  \biggl(\frac{D_k^{\min}}{\mass_k^2} + \frac{13}{48}\biggr)
		<\frac{7}{16}
		< \frac{1}{2} - \frac{\sqrt{D_k}}{\mass_k}
		< \frac{1}{2}
	\end{align*}
	and thereby proves inequality \eqref{eq:boundsphik}. 
	In the following, we assume that $k$ is sufficiently large so that the latter is satisfied.

	According to \autoref{sec:A0}, $\phi_k(t) = e^{t + d_1^k}$ on $(-\infty,  t_{0,k} - \varepsilon_k]$ for some constant $d_1^k \in \R$ depending on $k$.  
	Integrating gives
	 \begin{align*}
		f_k(e^{(t_{0,k} - \varepsilon_k)/2})
		&= \frac{1}{\sqrt{2}}\int_{- \infty}^{t_{0,k} - \varepsilon_k} \sqrt{\phi_{k}^\prime(s)}\D s 
		= \frac{1}{\sqrt{2}}\int_{- \infty}^{t_{0,k} - \varepsilon_k} e^{(s + d_1^k)/2}\D s 
		= \frac{2}{\sqrt{2}} e^{(t_{0,k} - \varepsilon_k + d_1^k)/2} \\
		&= \sqrt{2 \phi_k(t_{0,k} - \varepsilon_k)} 
		< \sqrt{2 \phi_k(t_{0,k} + \varepsilon_k)} 
		< \sqrt{\mass_k} = \frac{1}{\sqrt{k}}
	\end{align*}
	because of the monotonicity of $\phi_k$ and inequality \eqref{eq:boundsphik}.
	On the other hand, Hölder's inequality yields
	\begin{align*}
		f_k(e^{(t_{0,k} + \varepsilon_k)/2}) - f_k(e^{(t_{0,k} - \varepsilon_k)/2}) 
		&= \frac{1}{\sqrt{2}}\int_{t_{0,k} - \varepsilon_k}^{t_{0,k} + \varepsilon_k} \sqrt{\phi_{k}^\prime(s)}\D s \\
		&\leq \frac{1}{\sqrt{2}} \sqrt{2\varepsilon_k}\biggl(\int_{t_{0,k} - \varepsilon_k}^{t_{0,k} + \varepsilon_k} \phi_{k}^\prime(s)\D s \biggr)^{1/2} \\
		&= \sqrt{\varepsilon_k} \sqrt{\phi_k(t_{0,k} + \varepsilon_k) - \phi_k(t_{0,k} - \varepsilon_k)} \\
		&< \sqrt{\varepsilon_k \frac{\mass_k}{2}} = \sqrt{\frac{\varepsilon_k}{2k}}.
	\end{align*}
	Without loss of generality, we can assume that $\varepsilon_k \leq 1$ for all $k \in \N$. This allows us to conclude 
	\begin{align*}
		f_k(e^{(t_{0,k} + \varepsilon_k)/2}) = f_k(e^{(t_{0,k} - \varepsilon_k)/2}) + \bigl(f_k(e^{(t_{0,k} + \varepsilon_k)/2}) - f_k(e^{(t_{0,k} - \varepsilon_k)/2}) \bigr)
		< \frac{1}{\sqrt{k}} + \frac{1}{\sqrt{2k}} \rightarrow 0
	\end{align*}
	as $k \rightarrow \infty$.

	As for the second limit, 
	we use \eqref{eq:boundsphik}, which in particular implies $t_k^{(1/2)} > t_{0,k} + \varepsilon_k$, and the asymptotics of $D_k$ to compute
	 \begin{align*}
	f_k(e^{t_k^{(1/2)}/2})
	&= \frac{1}{\sqrt{2}}\int_{- \infty}^{t_k^{(1/2)}} \sqrt{\phi_{k}^\prime(s)}\D s 
	\geq  \frac{1}{\sqrt{2}}\int_{t_{0,k}+ \varepsilon_k}^{t_k^{(1/2)}} \sqrt{\phi_{k}^\prime(s)}\D s \\
	&= \frac{1}{\sqrt{2}}\int_{\phi_{k}(t_{0,k} + \varepsilon_k)}^{\phi_{k}(t_k^{(1/2)})} \sqrt{\phi_{k}^\prime} \frac{\D \phi_{k}}{\phi_{k}^\prime} 
	= \frac{1}{\sqrt{2}}\int_{\phi_{k}(t_{0,k} + \varepsilon_k)}^{\frac{\mass_k}{2}} \sqrt{\frac{\phi_{k}}{F_{D_k}(\phi_{k})}} \D \phi_{k} \\
	&\geq  \frac{1}{\sqrt{2}} \int_{\frac{7}{16}  \mass_k }^{\frac{\mass_k}{2} - \sqrt{D_k}} \sqrt{\frac{\phi_{k}}{(\phi_{k} - \frac{\mass_k}{2})^2 + D_k}} \D \phi_{k} 
	\geq  \frac{1}{\sqrt{2}} \int_{\frac{7}{16}  \mass_k}^{\frac{\mass_k}{2} - \sqrt{D_k}} \sqrt{\frac{\frac{7}{16}  \mass_k }{2(\phi_{k} - \frac{\mass_k}{2})^2}} \D \phi_{k} \\
	&= \frac{\sqrt{7  \mass_k}}{8} \int_{\frac{7}{16}  \mass_k}^{\frac{\mass_k}{2} - \sqrt{D_k}} {\frac{1}{\frac{\mass_k}{2} - \phi_{k} }} \D \phi_{k} 
	= - \frac{\sqrt{7  \mass_k}}{8} \biggl(\log \sqrt{D_k} - \log \biggl(\frac{\mass_k}{2} - \frac{7}{16}  \mass_k \biggr)\biggr) \\
	&= - \frac{\sqrt{7  \mass_k}}{8}\biggl(\log \frac{\sqrt{D_k}}{\mass_k} - \log \frac{1}{16}\biggr) 
	\sim  - \frac{\sqrt{7  \mass_k}}{8} \biggl(\log\frac{ \pi }{2 \lvert t_{0,k}\rvert} +\log 16 \biggr) \\
	&\sim  \frac{\sqrt{7  \mass_k}}{8}\log  \lvert t_{0,k}\rvert
	= \frac{\sqrt{7 k}}{8}
	\rightarrow \infty
\end{align*}
as $k \rightarrow \infty$. 
\end{proof}

We are now in a position to discuss the convergence of the constructed sequence.

\begin{lemma}\label{lemma:GHconvergence}
	If we take all basepoints to be the origin, the sequence $(g_k)_k$ converges in the pointed Gromov--Hausdorff sense to the half-line $[0, \infty)$ equipped with $ d_{[0, \infty)} \coloneqq d_{\eucl}\vert_{[0, \infty)}$,
	\begin{align*}
		(\C^2, d_{g_k}, 0) \stackrel{\mathrm{pGH}}{\longrightarrow} \bigl([0, \infty), d_{[0, \infty)}, 0\bigr).
	\end{align*}
\end{lemma}
 
 \begin{proof}
 	For $P > 0$ and $k \in \N$, we define
 	\begin{align*}
 		h^P_k : \bigl(\overline{{B}^{g_k}_P(0)} , d_{g_k}\bigr) \rightarrow \bigl([0, \infty), d_{[0, \infty)}\bigr), \qquad z \mapsto r_k(z) \coloneqq d_{g_k}(0, z),
 	\end{align*}
 	which satisfies $h^P_k(0) = 0$ and $h^P_k(\overline{{B}^{g_k}_P(0)} ) = [0, P]$.
 	To prove the claim, it suffices to show that for each $P > 0$, the distortion of $h^P_k$ given by
 	\begin{align*}
 		\dis h^P_k 
 		= \sup_{z, w \in \overline{{B}^{g_k}_P(0)} } \big\vert \vert r_k(z) - r_k(w) \vert - d_{g_k}(z,w) \big\vert
 	\end{align*}
 	tends to $0$
 	as $k \rightarrow \infty$
 	(\cf \autocite[\Def 8.1.1]{BuragoMetricGeometry}).

 	By definition of $r_k$, the supremum in the preceding formula can be equivalently taken over $\overline{{B}^{g_k}_P(0)} \setminus \{0\}$.
 	Thus, let $z, w \in \overline{{B}^{g_k}_P(0)} \setminus \{0\}$ be arbitrary and 
 	pick $\tilde{w} \in (0, \infty) \cdot z$ such that $r_k(w)= r_k(\tilde{w})$. 
 	As the metrics are radially symmetric, rays from the origin are length-minimizing geodesics and $d_{g_k}(z, \tilde{w}) = \lvert r_k(z) -r_k(\tilde{w}) \rvert= \lvert r_k(z) -r_k({w})  \rvert$. Consequently, 
 	\begin{align*}
 		\big\vert \vert r_k(z) - r_k(w) \vert - d_{g_k}(z,w) \big\vert  \leq d_{g_k}(w, \tilde{w})
 	\end{align*}
 	and 
 	\begin{align*}
 		\dis h^P_k
 		= \sup_{\rho \in [0, P]} \diam_{g_k} \{z \in \C^2 \mid r_k(z) = \rho\}. 
 	\end{align*}
 	
 	Let $T \in \R$.
 	Because of the spherical symmetry of $g_{k}$, 
 	 the shortest path between the origin and $(e^{T/2},0) \in \C^2$ is $\gamma_T : [-\infty,T] \rightarrow \C^2, s \mapsto (e^{s/2},0)$. As $g_k$ takes the form
 	 \begin{align*}
 	 	\bigl((g_{k})_{\alpha \bar{\beta}}\bigr)_{\alpha, \beta} = e^{-t} \begin{pmatrix}
 	 		\phi_{k}^{\prime} & 0 \\ 0 & \phi_{k}
 	 	\end{pmatrix}
 	 \end{align*}
 	 on $\{z_2 = 0\}$, this path satisfies
 	\begin{align*}
 		g_{k}(\dot{\gamma}_T, \dot{\gamma}_T)(s) = \biggl(\frac{1}{2} e^{s/2}\biggr)^2 2 (g_{k})_{1 \bar{1}}\vert_{t = s} = \frac{1}{2} \phi_{k}^\prime(s),
 	\end{align*}
 	which yields
 	\begin{align*}
 		r_k\bigl((e^{T/2},0)\bigr)
 		&= \int_{- \infty}^{T} \sqrt{g_{k}(\dot{\gamma}_T, \dot{\gamma}_T)(s)} \D s
 		= \frac{1}{\sqrt{2}}\int_{- \infty}^{T} \sqrt{\phi_{k}^\prime(s)}\D s.
 	\end{align*}
 	More generally, we get
 	 \begin{align*}
 		r_k = f_k(r) = \frac{1}{\sqrt{2}} \int^{\log r^2}_{- \infty} \sqrt{\phi_k^\prime(s)} \D s
 	\end{align*}
 	by radial symmetry.

 	The function $f_k: (0, \infty) \rightarrow (0 , \infty)$ is increasing, bijective and satisfies $\lim_{r \rightarrow 0^+}f_k(r) = 0$. 
 	In particular, spheres centered at $0$ with respect to $g_k$ are also spheres with respect to $g_{\eucl}$ and vice versa:
 	\begin{align*}
 		\{z \in \C^2 \mid r_k(z) = \rho\} =  \{z \in \C^2 \mid r(z) = f_k^{-1}(\rho)\}.
 	\end{align*}
 	 	Therefore, \autoref{lemma:diamEuclSphere} yields
 	\begin{align*}
 		d_k(\rho) \coloneqq
 		\diam_{g_k} \{z \in \C^2 \mid r_k(z) = \rho\}
 		\leq \sqrt{\frac{65}{2} \pi^2 \phi_k \bigl(\log f_k^{-1}(\rho)^2\bigr)
 			+ 32 \pi^2 \phi_k^\prime\bigl(\log f_k^{-1}(\rho)^2\bigr)}
 	\end{align*}
 	for every $\rho > 0$.

	We now focus on a fixed $\rho > 0$. 
	Let $\sigma_k^\pm = \pm \frac{\mass_k}{48}$, $D_k^\pm$ and $\delta_k \in (0, \frac{1}{2})$ be as in the proof of \autoref{lemma:globalKähler} for
	$(\mass_k, t_{0,k}) = (\frac{1}{k}, -e^k)$, 
	and write $D_k^{\max} \coloneqq \max \{D_k^+, D_k^-\}$ and $\sigma_k^{\max} \in \{\sigma_k^+, \sigma_k^-\}$ for the corresponding parameter.
	According to \autoref{lemma:limitfk}, $\rho$ satisfies
 	\begin{align*}
 		f_k(e^{(t_{0,k} + \varepsilon_k)/2}) \leq \rho \leq f_k(e^{t_k^{(1/2)}/2}) \iff t_{0,k} + \varepsilon_k \leq \log f_k^{-1}(\rho)^2 \leq t_k^{(1/2)}
 	\end{align*}
 	for $k$ sufficiently large. Hence, 
 	\begin{align*}
 		0 < \phi_k \bigl(\log f_k^{-1}(\rho)^2\bigr) \leq \phi_k(t_k^{(1/2)}) = \frac{\mass_k}{2} = \frac{1}{2k} \rightarrow 0
 	\end{align*}
 	as $k \rightarrow \infty$,
 	and 
 	\begin{align*}
 		\phi_k \bigl(\log f_k^{-1}(\rho)^2\bigr) 
 		&\geq \phi_k(t_{0,k} + \varepsilon_k)
 		> \phi_{D_k^{\max}} (t_{0,k} + \varepsilon_k) 
 		> (1 - \delta_k) \biggl( \frac{D_k^{\max}}{\mass_k} + \frac{\mass_k}{4} + \sigma_k^{\max}\biggr) \\
 		&>  \frac{\mass_k}{2} \biggl( \frac{D_k^{\max}}{\mass_k^2} + \frac{11}{48}\biggr).
 	\end{align*}
 	This further implies
 	\begin{align*}
 		0 < \phi_k^\prime \bigl(\log f_k^{-1}(\rho)^2\bigr) 
 		&= \frac{F_{D_k}\bigl(\phi_k \bigl(\log f_k^{-1}(\rho)^2\bigr)\bigr)}{\phi_k \bigl(\log f_k^{-1}(\rho)^2\bigr)} 
 		= \phi_k \bigl(\log f_k^{-1}(\rho)^2\bigr) - \mass_k + \frac{\frac{\mass_k^2}{4} + D_k}{\phi_k \bigl(\log f_k^{-1}(\rho)^2\bigr)} \\*
 		&< - \frac{\mass_k}{2} + \frac{\frac{\mass_k}{4} +\frac{D_k^{\max}}{\mass_k}}{\frac{1}{2} \Bigl( \frac{D_k^{\max}}{\mass_k^2} + \frac{11}{48}\Bigr)}
 		\rightarrow 0
 	\end{align*}
 	as $k \rightarrow \infty$ because $ D_k^{\max} \sim\frac{\mass_k^2 \pi^2 }{4t_{0,k}^2}$ by \autoref{lemma:c2Kähler}.
 	As a result, $d_k(\rho)\rightarrow 0$ as $k \rightarrow \infty$.

 It remains to show that this convergence is uniform on $[0, P]$.
To see this, let $\varepsilon > 0$ and choose a partition $0 = \rho_0 < \rho_1 < \dots < \rho_{N+1} = P$ of $[0, P]$ such that $\rho_{l+1} - \rho_l < \frac{\varepsilon}{4}$. 
As $N$ is finite and $d_k(\rho_l)\rightarrow 0$ as $k \rightarrow \infty$ for each $l$, there exists $K \in \N$ such that $d_k(\rho_l) < \frac{\varepsilon}{2}$ for all $k \geq K$ and $l \in \{0, \dots, N+1\}$. 
On the other hand, we have $d_k(\rho) \leq d_k(\rho^\prime) + 2 (\rho - \rho^\prime)$ for all $0 \leq \rho^\prime \leq \rho$ by radial symmetry.
Combining this gives 
\begin{align*}
	\dis h^P_k
	&= \sup_{\rho \in [0, P]} d_k(\rho)
	= \max_{l \in \{0, \dots, N\}} \sup_{\rho \in [\rho_l, \rho_{l+1}]} d_k(\rho)
	\leq \max_{l \in \{0, \dots, N\}}\bigl( d_k(\rho_l) + 2 (\rho_{l+1} - \rho_l)\bigr) \\*
	&< \frac{\varepsilon}{2} + \frac{\varepsilon}{2} = \varepsilon
\end{align*}
for all $k \geq K$.
In other words, $\dis h^P_k \rightarrow 0$ as $k \rightarrow \infty$ for all $P > 0$, and, consequently, 
\begin{align*}
	(\C^2, d_{g_k}, 0) \stackrel{\mathrm{pGH}}{\longrightarrow} \bigl([0, \infty), d_{[0, \infty)}, 0\bigr),
\end{align*}
which concludes the proof.
 \end{proof}
 
 Since the pointed Gromov--Hausdorff limit is unique up to isometry (cf., \eg \autocite[\Thm 8.1.7]{BuragoMetricGeometry}), the preceding lemma rules out convergence of $(\C^2, d_{g_k}, 0)$ to Euclidean space and thereby proves
 \autoref{thm:tentacleex}.

 \section{Sketch of the excised sets in \autoref{thm:Johan}}\label{sec:excision}
 
 In this section, we describe what the excised sets $Z_k$ in \autoref{thm:Johan} could look like for the sequence constructed above. 
 As we believe that a complete formalization of this part would not provide much additional insight, we rely partly on heuristics.
 
Set $z^* \coloneqq (e^{t^*/2}, 0) \in \C^2$ for some fixed $t^* \in \R$.
 For $\alpha > 0$ and $k \in \N$, we define $t_k^{(\alpha)} \coloneqq \phi_k^{-1}(\alpha \mass_k)$ and
 \begin{align*}
 	Z_k^{(\alpha)} \coloneqq \{z \in \C^2 \mid \phi_k(\log \lvert z \rvert^2) < \alpha \mass_k\}
 	= \{z \in \C^2 \mid \lvert z \rvert < e^{t_k^{(\alpha)} /2}\}.
 \end{align*}
 We will show that these sets satisfy 
  \begin{align*}
 	 	d_{g_k}(z^*, \partial Z_k^{(\alpha)}) =
 	 	d_{g_k}\bigl(z^*, (e^{t_k^{(\alpha)} /2}, 0)\bigr)
 	 	\rightarrow
 	 	\begin{cases}
 		 		\infty, & 0 < \alpha \leq \frac{1}{2},\\
 		 		d_\eucl(z^*, 0) = e^{t^*/2}, & \alpha > \frac{1}{2}.
 		 	\end{cases}
 	 \end{align*}
  Although this does not directly prove that $(\C^2 \setminus Z_k^{(\alpha)}, \hat{d}_{g_k})$ converges to Euclidean space in the pointed Gromov--Hausdorff sense whenever $\alpha > \frac{1}{2}$, it provides strong evidence for it and we consider it sufficient for our purposes.
  As for the area estimate, we will see that  
  \begin{align*}
  	\area_{g_k}(\partial Z_k^{(\alpha)}) \sim 4 \sqrt{2} \pi^2 \sqrt{\alpha} \biggl(\alpha - \frac{1}{2}\biggr) \mass_k^{3/2}
  \end{align*}
  for $\alpha > \frac{1}{2}$,
  which implies $\area_{g_k}(\partial Z_k^{(\alpha)}) \leq c_\alpha \m(g_k)^{3/2} $ for some constant $c_\alpha > 0$. Then for every function $\xi$ as in \autoref{thm:Johan}, there exists $k_\xi \in \N$ such that $\xi(\m(g_k)) \leq \frac{1}{c_\alpha}$ for all $k \geq k_\xi$, which verifies the claimed area estimate. 

In preparation for the proof, we first show that 
$\phi_k(t^*) > \frac{\mass_k}{2}$ for $k$ sufficiently large, and 
$\phi_k(t^*) \rightarrow \frac{1}{2}e^{t^*}$ as $k \rightarrow \infty$. 
  Let $K = K(t^*) \in \N$ be such that 
 $t^* \geq t_{0,k} + \varepsilon_k$ and $D_k \leq (\frac{1}{2}e^{t^*})^2$ for all $k \geq K$. 
If $\phi_{k}(t) = \frac{\mass_k}{2}$ for some $t \geq t^*$ and $k \geq K$, then \eqref{eq:implicitphir} yields the contradiction
  \begin{align*}
 	t + \log \frac{1}{2} = \frac{1}{2} \log D_k - \frac{\mass_k}{2 \sqrt{D_k}}  \frac{\pi}{2} < \log \sqrt{D_k} \leq 	t^* + \log \frac{1}{2}.
 \end{align*}
 Since $\phi_{k}$ is smooth and tends to $\infty$ as $t \rightarrow \infty$, this implies $\phi_{k}(t) > \frac{\mass_k}{2}$ for all $t \geq t^*$ and $k \geq K$. 
In the following, we always assume that the latter is satisfied. 
 
Formula \eqref{eq:implicitphir} also gives
 \begin{align*}
 		t^* + \log \frac{1}{2} \leq \frac{1}{2} \log F_{D_k}(\phi_k)
 		&\Longrightarrow \frac{1}{2} e^{t^*} \leq \sqrt{\biggl(\phi_k - \frac{\mass_k}{2}\biggr)^2 + D_k} \leq \phi_k - \frac{\mass_k}{2} + \sqrt{D_k} \\
 		&\Longrightarrow \phi_k(t^*) \geq \frac{1}{2} e^{t^*}  + \frac{\mass_k}{2} - \sqrt{D_k}
 \end{align*}
and 
\begin{align*}
	&t^* + \log \frac{1}{2} \geq \log \biggl(\phi_k - \frac{\mass_k}{2}\biggr) - \frac{\mass_k}{2 \bigl(\phi_k - \frac{\mass_k}{2} \bigr)} \\
	&\qquad \Longrightarrow \frac{\mass_k}{2 \bigl(\phi_k - \frac{\mass_k}{2} \bigr)} \geq \log \frac{\phi_k - \frac{\mass_k}{2}}{\frac{1}{2}e^{t^*}}
	\geq 1 -  \frac{\frac{1}{2}e^{t^*}}{\phi_k - \frac{\mass_k}{2}} \\
	&\qquad \Longrightarrow \phi_k(t^*) \leq \frac{1}{2} e^{t^*}  + \mass_k
\end{align*}
because $\arctan x - \frac{\pi}{2} = - \arctan \frac{1}{x} \geq - \frac{1}{x}$ and $\log x \geq 1 - \frac{1}{x}$ for all $x > 0$.
Therefore, 
\begin{align*}
	\frac{1}{2} e^{t^*}  
	\leq \liminf_{k \rightarrow \infty} \phi_k(t^*) 
	\leq \limsup_{k \rightarrow \infty} \phi_k(t^*) 
	\leq \frac{1}{2} e^{t^*}
\end{align*}
shows that $\phi_k(t^*) \rightarrow \frac{1}{2}e^{t^*}$ as $k \rightarrow \infty$.

Now suppose $0 < \alpha \leq \frac{1}{2}$. 
For $k$ sufficiently large, we have $\alpha \mass_k \leq \frac{\mass_k}{2} < \mass_k \leq \frac{1}{4}e^{t^*} \leq \phi_k(t^*)$ and $\phi_{k}(t_{0,k} + \varepsilon_k) < \frac{\mass_k}{2}$. Using this, we compute
\begin{align*}
	d_{g_k}\bigl(z^*, (e^{t_k^{(\alpha)} /2}, 0)\bigr)
	&= \frac{1}{\sqrt{2}}\int_{t_k^{(\alpha)} }^{t^*} \sqrt{\phi_{k}^\prime(s)}\D s 
	\geq \frac{1}{\sqrt{2}}\int_{t_k^{{(1/2)}} }^{t_k^{(1)}} \sqrt{\phi_{k}^\prime(s)}\D s \\
		&=  \frac{1}{\sqrt{2}} \int_{\frac{\mass_k}{2}}^{\mass_k} \sqrt{\frac{\phi_{k}}{F_{D_k}(\phi_{k})}} \D \phi_{k} 
	\geq \frac{\sqrt{\mass_k}}{2} \int_{\frac{\mass_k}{2}}^{\mass_k}\frac{1}{ \sqrt{(\phi_{k} - \frac{\mass_k}{2})^2 + D_k}} \D \phi_{k} \\
	&= \frac{\sqrt{\mass_k}}{2} \biggl( \log \biggl( \sqrt{\frac{\mass_k^2}{4} + D_k} + \frac{\mass_k}{2}\biggr) - \log \sqrt{D_k}\biggr) \\
	&\geq \frac{\sqrt{\mass_k}}{2} \log \frac{\mass_k}{\sqrt{D_k}}
	\sim \frac{\sqrt{\mass_k}}{2} \log \frac{2 \lvert t_{0,k} \rvert}{\pi} 
	\sim \frac{\sqrt{\mass_k}}{2} \log \lvert t_{0,k} \rvert \\
	&= {\frac{\sqrt{k}}{2}} \rightarrow \infty
\end{align*}
as $k \rightarrow \infty$. \clearpage

For the other case, fix $\alpha > \frac{1}{2}$ and let $k$ be sufficiently large such that $\phi_{k}(t_{0,k} + \varepsilon_k) < \frac{\mass_k}{2} < \alpha \mass_k \leq \frac{1}{4}e^{t^*} \leq \phi_k(t^*) \leq \frac{3}{4}e^{t^*}$. Then 
\begin{align*}
	d_{g_k}\bigl(z^*, (e^{t_k^{(\alpha)} /2}, 0)\bigr)
	&= \frac{1}{\sqrt{2}}\int_{t_k^{(\alpha)} }^{t^*} \sqrt{\phi_{k}^\prime(s)}\D s 
	= \frac{1}{\sqrt{2}}\int_{\alpha \mass_k}^{\phi_k(t^*)} \sqrt{\frac{\phi_{k}}{F_{D_k}(\phi_{k})}} \D \phi_{k} \\*
	&=  \frac{1}{\sqrt{2}}\int_0^\infty \mathbf{1}_{[\alpha \mass_k, \phi_k(t^*)]} \sqrt{\frac{\phi_{k}}{(\phi_{k} - \frac{\mass_k}{2})^2 + D_k}} \D \phi_{k}.
\end{align*}
The integrand in the last line converges pointwise to $\mathbf{1}_{[0, \frac{1}{2} e^{t^*}]} \phi_k^{- 1/2}$ as $k \rightarrow \infty$. On the other hand, we have 
\begin{align*}
	\sqrt{\frac{\phi_{k}}{(\phi_{k} - \frac{\mass_k}{2})^2 + D_k}}
	\leq  \frac{\sqrt{\phi_{k}}}{\phi_{k} - \frac{\mass_k}{2}}
	\leq \frac{\sqrt{\phi_{k}}}{\phi_{k} - \frac{\phi_k}{2 \alpha}}
	= \frac{2 \alpha}{2 \alpha - 1} \frac{1}{\sqrt{\phi_k}}
\end{align*}
for $\phi_k \geq \alpha \mass_k$, which yields the uniform bound 
\begin{align*}
	\mathbf{1}_{[\alpha \mass_k, \phi_k(t^*)]} \sqrt{\frac{\phi_{k}}{(\phi_{k} - \frac{\mass_k}{2})^2 + D_k}}
	\leq \mathbf{1}_{\left[0, \frac{3}{4} e^{t^*}\right]} \frac{2 \alpha}{2 \alpha - 1} \frac{1}{\sqrt{\phi_k}}.
\end{align*}
Since the right-hand side is integrable, dominated convergence gives the claimed limit,
\begin{align*}
	d_{g_k}\bigl(z^*, (e^{t_k^{(\alpha)} /2}, 0)\bigr)
	&\rightarrow \frac{1}{\sqrt{2}}\int_0^{\frac{1}{2}e^{t^*}} \frac{\D \phi_{k}}{\sqrt{\phi_k}} 
	= \frac{2}{\sqrt{2}} \sqrt{\frac{1}{2}e^{t^*}} = e^{t^*/2} = d_\eucl(z^*, 0).
\end{align*}

Finally, we turn to the area estimate. Let $z \in S^3_\rho$ for some $\rho > 0$.
Since $g_k(z,z) = 2 \phi_k^\prime(\log \rho^2)$ by \eqref{eq:coeffgk}, the volume form induced on $S^3_\rho$ by $g_k$ takes the form 
\begin{align*}
	\D A_{g_k} \vert_z 
	&= \frac{ z }{\sqrt{g_k(z,z)}} \lrcorner \D \vol_{g_k}
	= 4 e^{-2t} \phi_k \phi_k^\prime \frac{\rho}{\sqrt{2 \phi_k^\prime}} \biggl(\frac{z}{\rho} \lrcorner \D \vol_{\eucl}\biggr) \\
	&= 2 \sqrt{2} \rho^{-3} \phi_k \sqrt{\phi_k^\prime} \D A_\eucl \vert_z,
\end{align*}
	where $\phi_k$ and $\phi_k^\prime$ are evaluated at $t = \log \rho^2$.
	As a result, 
	\begin{align*}
		\area_{g_k} (S^3_\rho) = 2 \sqrt{2} \rho^{-3} \phi_k \sqrt{\phi_k^\prime} \area_{\eucl} (S^3_\rho) = 4 \sqrt{2} \pi^2 \phi_k \sqrt{\phi_k^\prime}.
	\end{align*}
Applying this to $\partial Z_k^{(\alpha)} $ for large $k$ so that $t_k^{(\alpha)} \geq t_{0,k} + \varepsilon_k$ yields 
\begin{align*}
	\area_{g_k}(\partial Z_k^{(\alpha)}) 
	&= 4 \sqrt{2} \pi^2 \phi_k (t_k^{(\alpha)})\sqrt{\phi_k^\prime (t_k^{(\alpha)})} 
	= 4 \sqrt{2} \pi^2  \alpha \mass_k \sqrt{\frac{\bigl(\alpha - \frac{1}{2}\bigr)^2 \mass_k^2 + D_k}{\alpha \mass_k}} \\
	&\sim 4 \sqrt{2} \pi^2  \sqrt{\alpha} \mass_k^{3/2} \sqrt{\biggl(\alpha - \frac{1}{2}\biggr)^2 + \frac{ \pi^2}{4 t_{0,k}^2}}
	\sim 4 \sqrt{2} \pi^2 \sqrt{\alpha} \biggl(\alpha - \frac{1}{2}\biggr) \mass_k^{3/2},
\end{align*}
which concludes the proof.

	\printbibliography	
	\noindent Mathematisches Institut, Universität Münster, Münster, Germany\\
	\emph{E-mail address}: \href{mailto:rmerkel@uni-muenster.de}{rmerkel@uni-muenster.de}

	\end{document}